\documentclass{amsart}
\usepackage{bbm}
\usepackage{graphicx}
\usepackage{verbatim}
\usepackage{stackrel}
\usepackage{color}
\graphicspath{{Fig/}}
\usepackage{amsmath}
\usepackage{amssymb}
\usepackage{array}
\usepackage{mathtools}
\usepackage{xcolor}
\usepackage{mathdots}
\usepackage{pgf}
\usepackage{bm}
\usepackage{hyperref}
\usepackage{CJKutf8}
\usepackage{mathrsfs}
\usepackage[thicklines]{cancel}

\usepackage{enumerate}

\usepackage{enumitem}

\usepackage{tikz}
\usetikzlibrary{shapes.arrows, fadings}
\usetikzlibrary{arrows}
\usetikzlibrary{matrix}
\usetikzlibrary{shapes,intersections}

\hypersetup{
    colorlinks,
    linkcolor={blue},
    citecolor={red},
}

\newtheoremstyle%
 {bluethm}%
 {}{}%
 {\color{blue!50!black!80!}\itshape}
 {}%
 {\color{blue!50!black!80!}\bfseries}%
 {\color{blue!50!black!80!}.}%
 { }{}

\newtheoremstyle%
 {greenthm}%
 {}{}%
 {\color{green!50!black!100!}\itshape}
 {}%
 {\color{green!50!black!100!}\bfseries}%
 {\color{green!50!black!100!}.}%
 { }{}

 \newtheoremstyle%
 {redthm}%
 {}{}%
 {\color{red}\itshape}
 {}%
 {\color{red}\bfseries}%
 {\color{red}.}%
 { }{}
\newtheorem{theorem}{Theorem}
\newtheorem{prop}{Proposition}
\newtheorem{lemma}{Lemma}

\newtheorem{example}{Example}

\theoremstyle{definition}

\newtheorem{defi}{Definition}

\newtheorem{remark}{Remark}

\newtheorem{assumption}{Assumption}[part]
\renewcommand{\theassumption}{\arabic{assumption}}

\def\N{{\mathbb N}}
\def\R{{\mathbb R}}

\newcommand{\diff}{\mathop{}\mathopen{}\mathrm{d}}
\newcommand\croc[1]{\left\langle #1\right\rangle}
\newcommand\ind[1]{\mathbbm{1}_{\left\{#1\right\}}}

\def\eps{\varepsilon}

\makeatletter
\newcommand{\proofstep}[1]{%
  \par
  \addvspace{\medskipamount}
  \textit{#1\@addpunct{.}}\enspace\ignorespaces
}
\makeatother

\title[Pathwise LDP in gelation]{Large-deviation bounds for cluster coagulation models with gelation}

\author{Mingcong Shi}
\email{smcandgzj@mail.ustc.edu.cn}

\author{Wen Sun}
\email{wensun.ustc@gmail.com}
\address        {School of Mathematical Sciences, University of Science and Technology of China, Jinzhai 96, 230026 Hefei}

\date{\today}
\keywords{Coagulation; gelation; Smoluchowski equation; cluster coagulation; Flory equation; large-deviation bounds.}

\begin{document}

\maketitle

We establish pathwise large-deviation bounds for the empirical-measure flow
of a cluster coagulation model with gelling kernels on a locally compact Polish
state space.  The cluster coagulation process is a measure-valued Markov jump
process in which pairs of clusters merge according to a time-inhomogeneous
kernel $K(t,x,y,\diff z)$.  Kernels with sufficiently rapid growth may
produce gelation: a phase transition in which mass escapes from the
finite-cluster population in finite time.  We first prove tightness of the empirical
flows and identify every subsequential limit as a solution of the
time-inhomogeneous multi-type Flory equation, which accounts for 
post-gelation dynamics through a conserved quantity.  We then establish
a global large-deviation upper bound on path space with a variational rate
function.  We also prove
a local lower bound near paths that are unique solutions of suitable tilted
Flory equations.  Finally, we give a convex-dual representation of the upper
rate function, including the possible singular part of the representing
measure.  Our results apply to locally compact
Polish state spaces with continuous mass functions, encompassing both 
the gelling and non-gelling regimes within a unified framework.

\bigskip

\section{Introduction}\label{section 1}
\enlargethispage{4pt}

Building on the theory of cluster coagulation and gelation developed in
\cite{norris1999,norris2000,Fournier2004,FOURNIER2009,
AIMconvergence}, we establish pathwise large-deviation bounds for
a time-inhomogeneous cluster coagulation process in the presence of
gelation.  The state space is a locally compact Polish space, and a
cluster may therefore carry attributes in addition to its mass.  The
macroscopic dynamics are described by a multi-type Flory equation whose
additional loss term records interactions between finite clusters and
the gel through a conserved quantity.  We prove a global
large-deviation upper bound on path space and a local pathwise lower
bound near trajectories that are unique solutions of suitable tilted
Flory equations.

\subsection{Coagulation, gelation, and the Flory equation}

The Smoluchowski coagulation equation~\cite{Smoluchowski} is the
classical mean-field model for a population of particles undergoing
successive binary mergers.  In its simplest form, particles of masses
$x$ and $y$ merge at rate $K(x,y)$ and produce a particle of mass
$x+y$.  For kernels with sufficiently rapid growth, solutions may lose
mass after a finite time.  This phenomenon, known as \emph{gelation},
represents the formation of clusters whose size is macroscopic relative
to the scale on which the finite-particle distribution is observed; see,
for example,
\cite{ziff1980,vandongen_ernst_1986,jeon1998,
escobedo_mischler_perthame_2002,Laurencot2015}.

The corresponding finite-particle model is the Marcus--Lushnikov
process~\cite{marcus1968,lushnikov1978}.  Starting with finitely many
particles, each unordered pair merges at a rate proportional to
$K(x,y)$ and inversely proportional to the system size.  Hydrodynamic
limits of this process have been studied under many assumptions; see
\cite{Aldous1999,norris1999,CEPEDA20111411} and the references therein.
In a gelling regime, however, the usual Smoluchowski equation does not
describe all post-gelation interactions: a macroscopic cluster can
continue to absorb finite clusters.  The resulting correction leads to
the Flory equation.  For classical mass coagulation with a
time-homogeneous kernel, a representative weak form is
\begin{equation}\label{eq:flory-intro-classical}
\begin{aligned}
\langle g,\mu_t\rangle
={}&\langle g,\mu_0\rangle
+\frac12\int_0^t\!\int_{[0,\infty)^2}
  [g(x+y)-g(x)-g(y)]K(x,y)
  \mu_s(\diff x)\mu_s(\diff y)\diff s\\
&-\int_0^t\!\int_{[0,\infty)}g(y)\,l(y)
  \bigl(\langle x,\mu_0\rangle-\langle x,\mu_s\rangle\bigr)
  \mu_s(\diff y)\diff s.
\end{aligned}
\end{equation}
Here $l(y)$ is the asymptotic rate at which a finite cluster of mass $y$
interacts with a large cluster.  The second line describes absorption of
finite clusters by the gel.  Such equations and their stochastic
approximations were studied in
\cite{norris2000,Fournier2004,FOURNIER2009}; see also
\cite{reza} for stochastic gelation estimates.

The cluster coagulation model permits cluster types in a general state
space $E$.  A kernel $K(t,x,y,\diff z)$ specifies both the rate at which
clusters $x$ and $y$ merge and the distribution of the resulting type
$z$.  Following Andreis, Iyer, and Magnanini~\cite{AIMconvergence}, the
post-gelation correction can be described by a conserved quantity
$\phi$.  In the present time-inhomogeneous setting, the defining
identity is required to hold along the entire evolution: for all
$r,s\in[0,T]$,
\[
\phi(s,z,q)=\phi(s,x,q)+\phi(s,y,q)
\quad\text{for }K(r,x,y,\cdot)\text{-a.e.\ }z.
\]
The parameter $s$ labels the conserved family, whereas $r$ is the actual
time at which a coagulation occurs.  The limiting weak equation is
\begin{equation}\label{eq:Flory-intro}
\begin{aligned}
\langle g,\mu_t\rangle
={}&\langle g,\mu_0\rangle
+\frac12\int_0^t\!\int_{E^2\times E}
L[g](x,y,z)\,K(s,x,y,\diff z)\,
\mu_s(\diff x)\mu_s(\diff y)\diff s\\
&-\int_0^t\!\int_E g(y)
\left[\int_E\phi(s,x,y)(\mu_0-\mu_s)(\diff x)\right]
\mu_s(\diff y)\diff s,
\end{aligned}
\end{equation}
where $L[g](x,y,z)=g(z)-g(x)-g(y)$.  The bracket in the second line is
nonnegative along a Flory solution and measures the conserved quantity
that has left the finite-cluster population.

Time-inhomogeneity is useful even when the original physical kernel is
time independent.  Large-deviation changes of measure replace $K$ by
tilted kernels whose densities depend on time, and a proof that is stable
under such tilts naturally requires a time-dependent formulation.

\subsection{Main results}

Let $E$ be a locally compact Polish space, let
$\bar E=E\cup\{\infty\}$ be its one-point compactification, and put
$X_T=[0,T]\times\bar E^2\times\bar E$.  The empirical measure of the
finite system is
\[
\mu_t^N=\frac1N\sum_{i=1}^{n_t}\delta_{x_i(t)}.
\]
The process and its generator are defined precisely in
Definition~\ref{def:process}. Assumption~\ref{assu:kernel} imposes
continuity, growth control, proximity, uniform compact-support properties, and
uniform tail estimates.  Assumption
\ref{assu:initial} supplies convergence and compact containment for the
initial measures.

For a path $\pi$, define the finite-cluster coagulation measure together
with its gel correction by
\begin{equation}\label{eq:Lambda-intro}
\begin{aligned}
\langle f,\Lambda_\pi^{K,T}\rangle
={}&\frac12\int_0^T\!\int_{E^2\times E}
 f(t,x,y,z)\,K(t,x,y,\diff z)\,
 \pi_t(\diff x)\pi_t(\diff y)\diff t\\
&+\int_0^T\!\int_E f(t,\infty,y,\infty)
 \left[\int_E\phi(t,x,y)(\pi_0-\pi_t)(\diff x)\right]
 \pi_t(\diff y)\diff t.
\end{aligned}
\end{equation}
For paths in the set $\mathscr A_\phi$ introduced in Section
\ref{section 2}, the two terms are absolutely integrable and the
right-hand side defines a finite positive measure.  If $g$ extends
continuously to $\bar E$, equation~\eqref{eq:Flory-intro} can be written
compactly as
\[
\langle g,\mu_t-\mu_0\rangle
=\langle L[\bar g],\Lambda_\mu^{K,t}\rangle.
\]

Our first result is a subsequential hydrodynamic limit.  Theorem
\ref{thm:LLN} proves tightness of $(\mu^N)$ and shows that every
subsequential limit is almost surely a continuous solution of the
time-inhomogeneous multi-type Flory equation.

For the upper bound, let
\[
\mathscr C_{\bar E}^{1,0}
=\left\{g\in\mathscr C_b^{1,0}([0,T]\times E):
g\text{ has an extension }\bar g\in
\mathscr C_b^{1,0}([0,T]\times\bar E)\right\}.
\]
The variational upper rate is
\begin{equation}\label{eq:rate-upper-intro}
\begin{aligned}
\mathcal R_{\rm upper}^K(\pi)
=\sup_{g\in\mathscr C_{\bar E}^{1,0}}
\bigg\{&\langle g_T,\pi_T\rangle-\langle g_0,\pi_0\rangle
-\int_0^T\langle\partial_tg_t,\pi_t\rangle\diff t\\
&-\langle e^{L[\bar g]}-1,\Lambda_\pi^{K,T}\rangle\bigg\}.
\end{aligned}
\end{equation}
Theorem~\ref{thm:LDP upper} states that, for every closed subset $C$ of
the path space,
\[
\limsup_{N\to\infty}\frac1N\log\mathbb P(\mu^N\in C)
\leq-\inf_{\pi\in C}
\{\iota_\nu(\pi_0)+\mathcal R_{\rm upper}^K(\pi)\}.
\]
The proof uses exponential martingales.  To establish exponential
tightness, we verify exponential compact containment and an Aldous
stopping-time condition at the exponential scale.

The lower bound is deliberately stated on a smaller class.  A strictly
positive bounded continuous function $\eta\in\mathscr C_b(X_T)$,
symmetric in $(x,y)$, tilts the finite-cluster kernel to $\hat\eta K$
and replaces the conserved quantity $\phi$ by
$\eta(t,\infty,y,\infty)\phi(t,x,y)$.  Such tilts preserve the standing
assumptions.  If a path $\pi$ is the unique solution of the resulting
tilted Flory equation, then every open neighborhood $O$ of $\pi$
satisfies
\begin{equation}\label{eq:rate-lower-intro}
\liminf_{N\to\infty}\frac1N\log\mathbb P(\mu^N\in O)
\geq-\langle\tau^*(\eta-1),\Lambda_\pi^{K,T}\rangle,
\end{equation}
where
\[
\tau^*(u)=
\begin{cases}
(1+u)\log(1+u)-u,&u>-1,\\
1,&u=-1,\\
+\infty,&u<-1.
\end{cases}
\]
Taking the infimum over all such uniqueness-admissible tilts gives
Theorem~\ref{thm:LDP lower}.  Thus the lower bound holds locally near
unique tilted solutions.  Existing uniqueness criteria, including the
eventual-conservativity condition in
\cite[Theorem~2.4]{AIMconvergence}, can be used to verify this hypothesis
in particular models; the assumptions of this paper alone do not imply
it in general.

Finally, Theorem~\ref{thm:another representation} identifies the convex
dual structure of the upper rate.  Because the dual of
$\mathscr C(X_T)$ consists of signed Radon measures, the general formula
allows a nonpositive component singular with respect to
$\Lambda_\pi^{K,T}$.  If the optimizing measure is absolutely continuous,
with density $z$, the representation reduces to the familiar entropy
form
\[
\mathcal R_{\rm upper}^K(\pi)
=\inf_z\langle\tau^*(z),\Lambda_\pi^{K,T}\rangle
\]
subject to the linear weak-evolution constraint specified in that
theorem.  Keeping the possible singular part in the general statement is
necessary for the conjugacy formula on continuous test functions.

\subsection{Related work and proof strategy}

Pathwise large deviations for interacting particle systems are commonly
proved through exponential martingales and changes of measure; classical
examples include exclusion and reaction--diffusion systems
\cite{kipnis1989hydrodynamics,jona1993}.  General treatments for
mean-field systems include
\cite{dawgar,djehiche1995,leonard1995large}.  For pure-jump
$k$-nary interactions, including nongelling coagulation models, a
martingale-measure approach was developed in \cite{ldpsun}.  The present
work adapts that approach to a setting in which a conserved quantity can
escape to the gel and re-enter the finite-particle dynamics through the
Flory correction.

A different large-deviation description of coagulation was obtained by
Andreis, K\"onig, Langhammer and Patterson~\cite{andreis2019largedeviations} using
empirical measures of coagulation histories and a Gibbs representation.
Their state variable and rate functional differ from the empirical
measure flow considered here.  Our use of the conserved-quantity
framework is based on the hydrodynamic theory of
\cite{AIMconvergence}.

Relative to these works, the present contribution has three distinct
parts.  First, the compactified measure $\Lambda_\pi^{K,T}$ places finite
coagulations and gel interactions in a common variational framework.
Second, the upper bound is global on path space and is supported by an
exponential-tightness argument for gelling kernels satisfying the stated
assumptions.  Third, the
change-of-measure argument gives a
lower bound near uniqueness-admissible tilted paths.  The third statement
is intentionally local: without a general uniqueness theorem for the
tilted Flory equation, these results do not constitute a full matching
large-deviation principle.

Two estimates drive the proofs.  First, moment bounds define a closed
set $\mathscr D_T^a$ containing every sufficiently large finite system.
Along uniformly weakly convergent sequences with a continuous limit and
the required initial-data compatibility, the map
\[
\pi\longmapsto\langle f,\Lambda_\pi^{K,t}\rangle
\]
is continuous uniformly in $t$.  The proof separates compact
coagulations, one-large/one-compact interactions, two-large interactions,
and the two parts of the gel correction.  The uniform compact-support
condition in Assumption~\ref{assu:kernel-4} is used for the compact part,
while proximity and tail assumptions control the other pieces.

Second, a marked counting measure retains the complete event
$(x,y,z)$.  Its compensator yields both the exponential martingales used
for the upper bound and the Radon--Nikodym density used for the lower
bound.
A diagonal estimate removes the difference between ordered product
measures and the finite-particle pair measure.  Under the tilted law,
the hydrodynamic theorem and uniqueness force concentration at the
selected path; reversing the change of measure then gives
\eqref{eq:rate-lower-intro}.

\subsection{Notation and organization}

For a Polish space $Y$, let $\mathscr M_f^+(Y)$ and $\mathscr M_f(Y)$
denote, respectively, the spaces of finite nonnegative Borel measures
and finite signed Radon measures, equipped with the weak topology.  We
write $\langle f,\mu\rangle=\int f\,\diff\mu$,
$\mathscr C_b(Y)$ for bounded continuous functions, and
$\mathscr C_c(Y)$ for continuous compactly supported functions.  The
spaces of continuous and c\`adl\`ag paths on $[0,T]$ are denoted by
$\mathscr C_T(Y)$ and $\mathscr D_T(Y)$, respectively; the latter carries
the $J_1$ Skorokhod topology.  Although coagulating pairs are unordered,
we write $E^2$ and attach an independent fair orientation mark to each
coagulation event when an ordered representation is needed.

Section~\ref{section 2} defines the process, the Flory equation, the
assumptions, and the rate functions.  Section~\ref{section 3} proves the
moment bounds, continuity estimate, diagonal estimate, and hydrodynamic
limit.  Section~\ref{section 4} constructs the marked martingales and
proves exponential tightness, the upper and local lower bounds, and the
dual representation of the upper rate.

\section{The model and main results}\label{section 2}
This section defines the stochastic cluster coagulation model and states the
main results.  The particle system is a measure-valued Marcus--Lushnikov jump
process with a time-inhomogeneous coagulation kernel.  Its deterministic
counterpart is the multi-type Flory equation, whose gel term records the loss
of finite-cluster mass through a conserved quantity.  We impose regularity,
growth, and compactness assumptions on the kernel and the initial measures.

Theorem~\ref{thm:LLN} proves tightness and identifies all subsequential limits
as solutions of the Flory equation.  Theorem~\ref{thm:LDP upper} gives the
large-deviation upper bound, while Theorem~\ref{thm:LDP lower} gives a local
lower bound near paths that are unique solutions of suitable tilted Flory
equations.  Theorem~\ref{thm:another representation} gives a convex-dual
representation of the upper rate.

\subsection{Dynamics of the cluster coagulation model}
To analyze the cluster coagulation model, we first define the measure-valued Markov process $(\mu_t^N)_{t\in[0,T]}$ via its infinitesimal generator.

\begin{defi}[The time-inhomogeneous cluster coagulation process]\label{def:process}
   Let $E$ be a locally compact Polish space with a continuous mass
   function $m:E\to[0,\infty)$.  Let
   $K:[0,T]\times E^2\to \mathscr{M}_f^+(E)$ be a measurable kernel,
   symmetric in its two input variables, and assume that coagulation
   conserves mass:
   \[
   K\bigl(t,x,y,\{z:m(z)=m(x)+m(y)\}\bigr)=\bar K(t,x,y),
   \qquad \bar K(t,x,y):=K(t,x,y,E).
   \]
   Given an initial configuration of clusters $x_1,\ldots,x_n\in E$
   with positive total mass, we write
   $N=\sum_{i=1}^n m(x_i)>0$ and call $N$ the system size.  The
   corresponding empirical measure is
   $\mu_0^N=N^{-1}\sum_{i=1}^n\delta_{x_i}$.

    The associated time-inhomogeneous cluster coagulation process $(\mu_t^N)_{t\in[0,T]}$ is a $\mathscr{M}_f^+(E)$-valued c\`{a}dl\`{a}g Markov process satisfying:

    \begin{enumerate}[label=(\roman*)]
        \item The state space is 
    \[
    \mathscr{M}_N^+(E) = \left\{ N^{-1} \sum_{i=1}^n \delta_{y_i} \bigg| \sum_{i=1}^n m(y_i) = N, n\in\N, y_i\in E,\forall 1\le  i\le n \right\}.
    \]

        \item Its infinitesimal generator acts on measurable functions  $\psi: \mathscr{M}_f^+(E)\to\R$ as follows: for any state $\mu=N^{-1}\sum_{i=1}^k\delta_{y_i}$,

        \begin{multline*}     
            L_t^{K,N}\psi(\mu) = \sum_{1\leq i<j\leq k} \int_E \{\psi[\mu+N^{-1}(\delta_{z}-\delta_{y_i}-\delta_{y_j})]-\psi[\mu]\}\frac{K(t,y_i,y_j,\diff z)}{N}\\
            = \frac{N}{2}\int_{E\times E\times E} \big(\psi[\mu+N^{-1}(\delta_{z}-\delta_x-\delta_y)]-\psi[\mu]\big) K(t,x,y,\diff z) \mu(\diff x) \left(\mu-\frac{\delta_x}{N}\right)(\diff y).
         \end{multline*}

    \end{enumerate}
\end{defi}

\subsection{Flory equation}

\begin{defi}[Conserved or sub-conserved quantities]\label{def:Conserved or sub-conserved quantities}
A function $\phi:[0,T]\times E\times E \to \R$ is said to be
conservative (for the time-inhomogeneous kernel $K$) if, for all
$r,s\in[0,T]$ and $x,y,q\in E$, for $K(r,x,y,\cdot)$-a.e. $z$,
\[
\phi(s,z,q)=\phi(s,x,q)+\phi(s,y,q).
\]
It is called sub-conservative if
\[
\phi(s,z,q)\leq \phi(s,x,q)+\phi(s,y,q).
\]
It is said to be doubly conservative (respectively, doubly
sub-conservative) if the analogous property also holds when a
coagulation is performed in the third argument; explicitly, one also
requires
\[
\phi(s,q,z)=\phi(s,q,x)+\phi(s,q,y)
\quad\text{(respectively, }\leq\text{)}
\]
for $K(r,x,y,\cdot)$-a.e. $z$. Thus the time parameter
$s$ labels a family of quantities, each of which is conserved along the
whole time-inhomogeneous evolution. Throughout the paper, the doubly
sub-conservative control functions are taken to be independent of $s$.
\end{defi}

\begin{example}
    \[\phi(t,x,y)=m(x)\]
    is a conserved quantity.
\end{example}

We now define solutions to the time-inhomogeneous multi-type Flory equation with coagulation kernel $K$ and the conserved quantity $\phi$.

\begin{defi}[The time-inhomogeneous multi-type Flory equation]\label{def:Flory}
Given a conserved quantity $\phi$, we say a map $t \mapsto \mu_t\in \mathscr M_f^+(E)$, $t\in[0,T]$, is a solution of the multi-type Flory equation with coagulation kernel $K$ and conserved quantity $\phi$ if the following are satisfied:
\begin{enumerate}[label=\arabic*.]
    \item for all Borel sets $A\subseteq E$ the map $t\mapsto\mu_t(A):[0,T]\to [0,\infty)$ is measurable;
    \item for all $g\in \mathscr{C}_c(E)$ and $t\in[0,T]$,
    \[
    \int_0^t\!\int_{E^2\times E}|L[g](x,y,z)|
    K(s,x,y,\diff z)\mu_s(\diff x)\mu_s(\diff y)\diff s<\infty
    \]
    and
    \[
    \int_0^t\!\int_{E^2}|g(y)|\phi(s,x,y)
    (\mu_0+\mu_s)(\diff x)\mu_s(\diff y)\diff s<\infty;
    \]
    \item for all $g\in \mathscr{C}_c(E)$ and $t\in[0,T]$,
    \begin{equation}\label{eq:multi-type Flory equation}
        \begin{aligned}
            \croc{g,\mu_t}=&\croc{g,\mu_0}+\frac{1}{2}\int_0^t \int_{E\times E\times E} L[g](x,y,z)K(s,x,y,\diff z)\mu_s(\diff x)\mu_s(\diff y)\diff s\\
            &-\int_0^t \int_E g(y) \left[\int_E \phi(s,x,y) [\mu_0(\diff x)-\mu_s(\diff x)]\right] \mu_s(\diff y)\diff s,
        \end{aligned}
    \end{equation}
    where $L[g](x,y,z)=g(z)-g(x)-g(y)$;
    \item for each $y\in E$ and $s,t\in[0,T]$ we have
    \[
    \int_E \phi(s,x,y)\mu_t(\diff x)\leq \int_E \phi(s,x,y)\mu_0(\diff x).
    \]
\end{enumerate}
\end{defi}
To simplify the proofs, we introduce the following measure on a
compactified space.
\begin{defi}
Let $X_T=[0,T]\times\bar E^2\times\bar E$.  For
$\pi\in\mathscr D_T(\mathscr M_f^+(E))$, suppose that
\begin{align*}
&\int_0^T\!\int_{E^2}\bar K(t,x,y)
  \pi_t(\diff x)\pi_t(\diff y)\diff t<\infty,\\
&\int_0^T\!\int_E
 \left|\int_E\phi(t,x,y)(\pi_0-\pi_t)(\diff x)\right|
 \pi_t(\diff y)\diff t<\infty.
\end{align*}
Then the following formula defines a finite signed Radon measure
$\Lambda_\pi^{K,T}$ on $X_T$: for every $f\in\mathscr C(X_T)$,
\begin{equation}\label{eq:equivalent to main continuous proposition}
    \begin{aligned}
        \croc{f,\Lambda_\pi^{K,T}}=&\frac{1}{2} \int_0^T\int_{E\times E\times E} f(t,x,y,z) K(t,x,y,\diff z) \pi_t(\diff x) \pi_t(\diff y) \diff t\\
        &+\int_0^T \int_{E} f(t,\infty,y,\infty)\left[\int_E \phi(t,x,y) [\pi_0(\diff x)- \pi_t(\diff x)]\right] \pi_t(\diff y) \diff t.
    \end{aligned}
\end{equation}
We denote by $\mathscr A_\phi$ the set of paths satisfying the two
integrability conditions above for which
\[
G_\pi(\diff t,\diff y):=
\left[\int_E\phi(t,x,y)(\pi_0-\pi_t)(\diff x)\right]
\pi_t(\diff y)\diff t
\]
is a nonnegative measure. Thus, for $\pi\in\mathscr A_\phi$,
$\Lambda_\pi^{K,T}$ is a finite positive measure.
For $t\leq T$, $\Lambda_\pi^{K,t}$ denotes the restriction of this
measure to $[0,t]\times\bar E^2\times\bar E$.
\end{defi}

Every $g\in\mathscr C_c(E)$ extends continuously to $\bar E$ by setting
$\bar g(\infty)=0$.  More generally, we use the same notation whenever
$g\in\mathscr C_b(E)$ has a continuous extension to $\bar E$.  In either
case
\[
L[\bar g](x,y,z)=\bar g(z)-\bar g(x)-\bar g(y)
\]
is continuous and bounded on $\bar E^2\times\bar E$.  With this notation,
the weak form of the multi-type Flory equation is
\begin{equation}\label{lambda-Flory}
    \croc{g,\mu_t}-\croc{g,\mu_0}-\croc{L[\bar{g}],\Lambda_{\mu}^{K,t}}=0,
    \qquad g\in\mathscr C_c(E).
\end{equation}
\subsection{Main results}
In the time-homogeneous case, suitable hypotheses imply concentration of
trajectories; see Assumptions~2.1 and~2.2 in \cite{AIMconvergence}.  The
large-deviation analysis below requires stronger uniform tail and compactness
conditions.

\begin{assumption}[Kernel]\label{assu:kernel}
    We impose the following conditions on the coagulation kernel $K$.
    \begin{enumerate}[label=\arabic*., ref=\theassumption.\arabic*]
        \item {Continuity:} $K(t,x,y,\cdot)$ is a continuous map from $[0,T]\times E^2$ to $\mathscr{M}_f^+(E)$.\label{assu:kernel-1}
        
        \item {Pointwise dominance:} There exists a continuous, doubly
        sub-conservative function $\phi_1$ such that
        $\bar K(t,x,y)\leq\phi_1(x,y)$ for every $t\in[0,T]$ and
        $x,y\in E$.\label{assu:kernel-2}
        
        \item \label{assu:kernel-3} {Proximity:} There exist a continuous conservative function $\phi$ and a continuous doubly sub-conservative function $\phi_2$ such that for an increasing collection of compact sets $(C_k)_{k\in\N}\subset E$ with $\bigcap_{k\in\N}\overline{C_k^c}=\emptyset$, we have for any compact $C'\subset E$
            \begin{equation*}
                \limsup_{k\to\infty} \sup_{t\in[0,T],x\in C_k^c,y\in C'} \frac{ |\bar{K}(t,x,y)-\phi(t,x,y)|}{\phi_2(x,y)}=0.
            \end{equation*} 
        
          \item { Compact support properties:}
          \begin{itemize}
         \item For any compact set $C\subset E$, there exists a compact set $C'\subset E\times E$ such that for any $t\in[0,T]$, if $K(t,x,y,C)>0$, then $(x,y)\in C'$. 
	         \item For every compact set $C^*\subset E\times E$, there
	         exists a single compact set $\widetilde C\subset E$ such that,
	         simultaneously for all $(x,y)\in C^*$ and $t\in[0,T]$,
	         the measure $K(t,x,y,\cdot)$ is supported on $\widetilde C$.
          \end{itemize}
	          Enlarging these compact sets if necessary, we may assume
	          $C'\supseteq C\times C$ and
	          $\widetilde C\times\widetilde C\supseteq C^*$.
          \label{assu:kernel-4}
        
        \item \label{assu:kernel-5} {Tail control:} There exists a continuous doubly sub-conservative function $\phi_3$ such that
        \begin{equation*}
            \limsup_{k\to\infty} \sup_{t\in[0,T],x\in C_k^c,y\in C_k^c} \frac{ \bar{K}(t,x,y)}{\phi_3(x,y)}=0,
        \end{equation*}
        and
        \begin{equation*}
            \limsup_{k\to\infty} \sup_{t\in[0,T],x\in C_k^c,y\in E} \frac{ \phi(t,y,x)}{\phi_3(y,x)}=0.
        \end{equation*}
    \end{enumerate}
     Replacing the control functions by their sum, we may and do use a
     single nonnegative, continuous, doubly sub-conservative function
     $\phi_0$ in Items~\ref{assu:kernel-2}, \ref{assu:kernel-3}, and
     \ref{assu:kernel-5}, with $0\leq\phi\leq\phi_0$.
\end{assumption}

\begin{assumption}[Initial conditions]\label{assu:initial}
    Let $K$ be a coagulation kernel satisfying Assumption~\ref{assu:kernel} for some control functions $\phi$ and $\phi_0$. We assume that the sequence of deterministic initial measures $(\mu_0^N)_{N>0}$ satisfies the following conditions:
    
    \begin{enumerate}[label=\arabic*., ref=\theassumption.\arabic*]
        \item \label{assu:initial-1}{Uniform bound:} The initial measures satisfy a boundedness condition with respect to $\phi_0$:
        \begin{equation*}
            \limsup_{N\to\infty} \int_{E\times E} \phi_0(x,y) \mu_0^N(\diff x)\mu_0^N(\diff y) < \infty.
        \end{equation*}

        \item \label{assu:initial-2}{Convergence:} There exists a  measure $\nu \in \mathscr{M}_f^+(E)$ such that:
        \begin{enumerate}[label=(\alph*)]
            \item $\mu_0^N \to \nu$ weakly in $\mathscr{M}_f^+(E)$ as $N \to \infty$;
            \item The total mass is strictly positive: $\langle m, \nu \rangle > 0$;
            \item For any compact set $C' \subset E$ and any $t \in[0,T]$, the following uniform convergence holds:
            \begin{equation*}
                \lim_{N\to\infty} \sup_{y \in C'} \left| \int_E \phi(t,x,y) \mu_0^N(\diff x) - \int_E \phi(t,x,y) \nu(\diff x) \right| = 0.
            \end{equation*}
        \end{enumerate}

        \item \label{assu:initial-3}{Tightness:} There exists a nonnegative, lower semicontinuous, doubly sub-conservative function $\phi^*$ such that for every $n \in \mathbb{N}$, the set
        \[
        \mathcal{E}_n^* = \left\{ \pi \in \mathscr{M}_f^+(E \times E) : \int_{E \times E} \phi^*(x,y) \pi(\diff x, \diff y) \leq n \right\}
        \]
        is compact in $\mathscr{M}_f^+(E \times E)$. Furthermore, the sequence $(\mu_0^N)$ satisfies:
        \begin{equation*}
            \limsup_{N\to\infty} \int_{E \times E} \phi^*(x,y) \mu_0^N(\diff x)\mu_0^N(\diff y) < \infty.
        \end{equation*}
    \end{enumerate}
\end{assumption}

We first state the subsequential hydrodynamic limit.  Its proof adapts
the compactness argument of \cite{AIMconvergence} and records the
additional uniform-in-time estimates needed here.

\begin{theorem}[Subsequential hydrodynamic limit]\label{thm:LLN}
Let $K$ satisfy Assumption~\ref{assu:kernel} with functions $\phi, \phi_0$, and let $(\mu_0^N)$ satisfy Assumption~\ref{assu:initial}. For each $N$, let $(\mu_t^N)_{t\in[0,T]}$ be the cluster coagulation process associated with $(K,\mu_0^N)$. Then:
\begin{enumerate}
\item The sequence $(\mu_t^N)_{t\in[0,T]}$ is tight in $\mathscr{D}([0,T],\mathscr{M}_f^+(E))$ endowed with the Skorokhod topology induced by the weak topology on $\mathscr{M}_f^+(E)$.
\item For any limit point $(\mu_t)_{t\in[0,T]}$ of a subsequence of $(\mu_t^N)$ as $N \to \infty$, the path $(\mu_t)_{t\in[0,T]}$ solves almost surely the multi-type Flory equation with coagulation kernel $K$, conserved quantity $\phi$, and initial condition $\nu$.
\end{enumerate}
\end{theorem}

\begin{remark}\label{rmk:extended LLN}
    Under our assumptions, the weak equation extends from
    $\mathscr C_c(E)$ to bounded continuous functions that extend
    continuously to $\bar E$.
\end{remark}

Let
\[
\mathscr C_{\bar E}^{1,0}:=
\left\{g\in\mathscr C_b^{1,0}([0,T]\times E):
g\text{ has an extension }\bar g\in
\mathscr C_b^{1,0}([0,T]\times\bar E)\right\}.
\]
For $\pi\in\mathscr A_\phi$, set
\begin{equation}\label{eq:upper-rate-compact}
\begin{aligned}
\mathcal R^K_{\rm upper}(\pi):=
\sup_{g\in\mathscr C_{\bar E}^{1,0}}
\bigg\{&\langle g_T,\pi_T\rangle-\langle g_0,\pi_0\rangle
-\int_0^T\langle\partial_tg_t,\pi_t\rangle\diff t\\
&-\left\langle e^{L[\bar g]}-1,\Lambda_\pi^{K,T}\right\rangle\bigg\}.
\end{aligned}
\end{equation}
We set $\mathcal R^K_{\rm upper}(\pi)=+\infty$ otherwise.  In
particular, the gel contribution in \eqref{eq:upper-rate-compact} is
\[
\int_0^T\int_E(1-e^{-g_t(y)})
\left[\int_E\phi(t,x,y)(\pi_0-\pi_t)(\diff x)\right]
\pi_t(\diff y)\diff t.
\]

\begin{theorem}[Large deviation upper bound]\label{thm:LDP upper}
Under Assumptions \ref{assu:kernel} and \ref{assu:initial}, for any closed set $C \subset \mathscr{D}_T(\mathscr{M}_f^+(E))$,
    \[
    \limsup_{N \to \infty} \frac{1}{N}\log\mathbb{P}(\mu^N \in C)
    \leq  -\inf_{\pi \in C} \left\{\iota_{\nu}(\pi_0) + \mathcal{R}^K_{\rm upper}(\pi) \right\},
    \]
where $\iota_\nu(\pi_0) = 0$ if $\pi_0 = \nu$ and $+\infty$ otherwise.
\end{theorem}

For a nonnegative $\eta\in\mathscr C_b(X_T)$ that is symmetric in
$(x,y)$, write
\begin{align*}
(\hat\eta K)(t,x,y,\diff z)
&:=\eta(t,x,y,z)K(t,x,y,\diff z),\\
\phi_\eta(t,x,y)
&:=\eta(t,\infty,y,\infty)\phi(t,x,y).
\end{align*}
For a path $\pi$ with $\pi_0=\nu$ and $\pi\in\mathscr A_\phi$, let
$\mathcal A^K(\pi)$ be the set of such $\eta$ for which $\pi$ solves
the multi-type Flory equation with kernel $\hat\eta K$ and conserved
quantity $\phi_\eta$; set $\mathcal A^K(\pi)=\emptyset$ otherwise.

\begin{defi}[Sub-Cameron--Martin space]\label{def:sub-CM space}
Define $H^K_0[\nu]$ as the set of paths
$\pi\in\mathscr C_T(\mathscr M_f^+(E))\cap\mathscr A_\phi$ such that
$\pi_0=\nu$ and $\mathcal A^K(\pi)\ne\emptyset$.
\end{defi}

\begin{defi}
Let $K$ satisfy Assumption~\ref{assu:kernel}. For
$\pi\in H_0^K[\nu]$, define the unrestricted candidate cost by
\[
\mathcal R^K_{\rm lower}(\pi):=
\inf_{\eta\in\mathcal A^K(\pi)}
\left\langle\tau^*(\eta-1),\Lambda_\pi^{K,T}\right\rangle.
\]
Here $\tau^*$ is the convex conjugate of $\tau(u)=e^u-u-1$ and is given by
    \[
    \tau^*(u)=\begin{cases}
        \begin{aligned}
            &(u+1)\log(u+1)-u \quad &&\text{if}\; u>-1,\\
            &1 &&\text{if}\; u=-1,\\
            &+\infty &&\text{if}\; u<-1.
        \end{aligned}
    \end{cases}
    \]
\end{defi}

This unrestricted entropy cost is the natural candidate for a matching
lower rate. The proof below, however, applies only to tilts for which the tilted
hydrodynamic equation has a unique solution.
The following stability property ensures that the tilted kernel satisfies the standing assumptions.
\begin{lemma}[Stability under bounded continuous tilts]
Suppose that $K$, $\phi$, and $(\mu_0^N)_{N>0}$ satisfy Assumptions~\ref{assu:kernel} and~\ref{assu:initial}. Let $\eta\in \mathscr{C}_b(X_T)$ be nonnegative and symmetric in $(x,y)$, and satisfy
\[ \inf_{X_T}\eta>0. \]
Then the kernel $\hat{\eta}K$, with conserved quantity $\phi_\eta$, and the initial sequence $(\mu_0^N)_{N>0}$ also satisfy Assumptions~\ref{assu:kernel} and~\ref{assu:initial}, with the same limiting initial measure $\nu$.
\end{lemma}

\begin{proof}
    Let
    \[
    0<c\leq \eta \leq C<\infty.
    \]
    The kernel $\hat{\eta}K$ with conserved quantity $\phi_\eta$ satisfies Assumption~\ref{assu:kernel-2}, \ref{assu:kernel-4}, and~\ref{assu:kernel-5} with control function $C\phi_0$. For the continuity condition, let
    \[
    (t_n,x_n,y_n)\longrightarrow(t,x,y)
    \]
    in $[0,T]\times E^2$. By Assumption~\ref{assu:kernel-1},
    \[
    K(t_n,x_n,y_n,\cdot)\Rightarrow K(t,x,y,\cdot).
    \]
    For any $h\in \mathscr C_b(E)$, we write
    \[
    \begin{aligned}
    &\int_E h(z)\eta(t_n,x_n,y_n,z)K(t_n,x_n,y_n,\diff z) - \int_E h(z)\eta(t,x,y,z)K(t,x,y,\diff z)\\
    =&\int_E h(z) \big[\eta(t_n,x_n,y_n,z)-\eta(t,x,y,z)\big]K(t_n,x_n,y_n,\diff z)\\
    &+ \int_E h(z)\eta(t,x,y,z) \big[K(t_n,x_n,y_n,\diff z)-K(t,x,y,\diff z)\big].
    \end{aligned}
    \]
    Since $X_T$ is compact and $\eta\in \mathscr C(X_T)$, uniform continuity gives
    \[
    \sup_{z\in\bar E}\big|\eta(t_n,x_n,y_n,z)-\eta(t,x,y,z)\big|\to 0.
    \]
    Moreover,
    \[
    \bar K(t_n,x_n,y_n)
    \to
    \bar K(t,x,y),
    \]
    so the total masses of $K(t_n,x_n,y_n,\cdot)$ are uniformly bounded. Hence the first term converges to zero. The second term converges to zero because
    \[
    z\longmapsto h(z)\eta(t,x,y,z)
    \]
    belongs to $\mathscr C_b(E)$. Therefore
    \[
    (\hat\eta K)(t_n,x_n,y_n,\cdot)
    \Rightarrow
    (\hat\eta K)(t,x,y,\cdot),
    \]
    which proves Assumption~\ref{assu:kernel-1} for the tilted kernel.

    For the proximity condition, note that
    \[
    \begin{aligned}
    &\left|\overline{\hat\eta K}(t,x,y)-\phi_\eta(t,x,y)\right|\\
    \leq& \int_E \left|\eta(t,x,y,z)-\eta(t,\infty,y,\infty)\right|K(t,x,y,\diff z)\\
    &+\eta(t,\infty,y,\infty)\big|\overline K(t,x,y)-\phi(t,x,y)\big|.
    \end{aligned}
    \]
    By the uniform continuity of $\eta$ on $X_T$ and Assumption~\ref{assu:kernel-4}, for every compact set $C'\subset E$ and every $\varepsilon>0$, there exists $k_0$ such that, for all $k\geq k_0$, $t\in[0,T]$, $x\in C_k^c$, and $y\in C'$,
    \[
    \big|\eta(t,x,y,z)-\eta(t,\infty,y,\infty)\big|<\varepsilon
    \]
    for $K(t,x,y,\diff z)$-almost every $z$. Together with the boundedness of $\eta$ and Assumption~\ref{assu:kernel-3} for the original kernel $K$, this proves Assumption~\ref{assu:kernel-3} for the tilted kernel.
    
    It remains to check the initial condition. Assumption~\ref{assu:initial-1}
    holds with control function $C\phi_0$;
    Assumption~\ref{assu:initial-2}(c) follows by multiplying its integrand by the bounded continuous function $y\mapsto\eta(t,\infty,y,\infty)$; and Assumption~\ref{assu:initial-2}(a)(b) and~\ref{assu:initial-3} is unchanged. Thus Assumption~\ref{assu:initial} also holds for the tilted kernel.
\end{proof}

We now encode this
restriction in the rate appearing in the theorem.

For $\pi\in H_0^K[\nu]$, let $\mathcal A^K_{\rm uniq}(\pi)$ be the
set of $\eta\in\mathcal A^K(\pi)$ such that
\begin{enumerate}[label=(\roman*)]
\item $\inf_{X_T}\eta>0$;
\item the Flory equation associated with $(\hat\eta K,\phi_\eta,\nu)$
has $\pi$ as its unique solution on $[0,T]$.
\end{enumerate}
Set
\[
\mathcal R^K_{\rm lower,uniq}(\pi):=
\inf_{\eta\in\mathcal A^K_{\rm uniq}(\pi)}
\left\langle\tau^*(\eta-1),\Lambda_\pi^{K,T}\right\rangle,
\]
with the convention $\inf\emptyset=+\infty$.

\begin{theorem}[Local large deviation lower bound]\label{thm:LDP lower}
Under Assumptions \ref{assu:kernel} and \ref{assu:initial}, for any open set $O \subset \mathscr{D}_T(\mathscr{M}_f^+(E))$,
\[
\liminf_{N \to \infty} N^{-1}\log\mathbb{P}(\mu^N \in O)
\ge -\inf_{\pi \in O\cap H_{0}^K[\nu]}
\mathcal R^K_{\rm lower,uniq}(\pi).
\]
Equivalently, if $\pi\in H_0^K[\nu]$ and
$\eta\in\mathcal A^K_{\rm uniq}(\pi)$, then every open neighborhood
$O$ of $\pi$ satisfies
\[
\liminf_{N\to\infty}N^{-1}\log\mathbb P(\mu^N\in O)
\ge-\left\langle\tau^*(\eta-1),\Lambda_\pi^{K,T}\right\rangle.
\]
\end{theorem}

\begin{remark}
Assumptions~\ref{assu:kernel} and \ref{assu:initial} alone do not provide
a general uniqueness theorem for the tilted Flory equation. Known
uniqueness results, such as \cite[Theorem~2.4]{AIMconvergence}, impose a
stronger eventual-conservativity condition. The restriction to
$\mathcal A^K_{\rm uniq}(\pi)$ is therefore substantive; if this set is
empty, the corresponding pointwise lower bound is vacuous.
\end{remark}

For a path for which $\Lambda_\pi^{K,T}$ is finite, define
\[
\gamma_T^K(\pi,g):=
\langle g_T,\pi_T\rangle-\langle g_0,\pi_0\rangle
-\int_0^T\langle\partial_tg_t,\pi_t\rangle\diff t
-\langle L[\bar g],\Lambda_\pi^{K,T}\rangle.
\]

\begin{theorem}[Representation of the upper rate function]\label{thm:another representation}
Let $K$ satisfy Assumption~\ref{assu:kernel}.  Suppose that
$\pi\in\mathscr A_\phi$ and
\[
\sup_{t\in[0,T]}\left\{
\langle1+\phi_0,\pi_t\otimes\pi_t\rangle+
\langle1+\phi_0,\pi_0\otimes\pi_t\rangle\right\}<\infty.
\]
Then $\Lambda:=\Lambda_\pi^{K,T}$ is a finite positive measure.  For a
finite signed Radon measure $\theta$ on $X_T$, write its Lebesgue
decomposition as $\theta=z\Lambda+\theta^s$ and set
\[
\mathcal J_\Lambda(\theta):=
\begin{cases}
\displaystyle\int_{X_T}\tau^*(z)\,\diff\Lambda
      +|\theta^s|(X_T),&\theta^s\leq0,\\[4pt]
+\infty,&\text{otherwise}.
\end{cases}
\]
Define
\[
\mathcal O_{\mathcal M}^K[\pi]:=
\left\{\theta\in\mathscr M_f(X_T):
\gamma_T^K(\pi,g)=\int_{X_T}L[\bar g]\diff\theta
\ \text{for every }g\in\mathscr C_{\bar E}^{1,0}\right\}.
\]
Then
\begin{equation}\label{eq:rate-representation}
\mathcal R^K_{\rm upper}(\pi)=
\inf_{\theta\in\mathcal O_{\mathcal M}^K[\pi]}
\mathcal J_\Lambda(\theta).
\end{equation}
In particular, if the infimum in \eqref{eq:rate-representation} is
unchanged when restricted to measures absolutely continuous with respect
to $\Lambda$ (for example, if it has an absolutely continuous minimizer),
then
\[
\mathcal R^K_{\rm upper}(\pi)=
\inf_{z\in\mathcal O^K[\pi]}\left\langle\tau^*(z),\Lambda\right\rangle.
\]
Here
\[
\mathcal O^K[\pi]:=\left\{z\in L^1(\Lambda):
\gamma_T^K(\pi,g)=
\langle z,L[\bar g]\rangle_{\Lambda}
\ \text{for every }g\in\mathscr C_{\bar E}^{1,0}\right\}.
\]
In particular, $\mathcal R^K_{\rm upper}(\pi)<\infty$ if and only if
$\mathcal O_{\mathcal M}^K[\pi]$ contains a measure $\theta$ with
$\mathcal J_\Lambda(\theta)<\infty$.
\end{theorem}

\subsection{Application to the Smoluchowski coagulation model}
Our framework includes the classical Smoluchowski coagulation model after
viewing its state space as $E=[0,\infty)$ and extending the kernel at zero
in a harmless way.  Its one-point compactification is then
$\bar E=[0,\infty]$.  We take
\[
K(t,x,y,\diff z)=K(t,x,y)\delta_{x+y}(\diff z).
\]
The assertion below is a specialization of the abstract results, not a
claim that every gelling kernel automatically satisfies all of
Assumption~\ref{assu:kernel}.

Suppose that $K$ is continuous and symmetric, and that its growth and
uniform tail limits verify Assumption~\ref{assu:kernel}; a standard model
case has
\[
K(t,x,y)\asymp xy^\alpha+x^\alpha y,
\qquad \alpha\in[0,1),
\]
as in \cite{Fournier2004,FOURNIER2009}.  When the following limit exists
with the uniformity required in Assumption~\ref{assu:kernel-3}, one may take
\[
\phi(t,x,y)=xl(t,y), \quad\textrm{where }l(t,y)=\lim_{x\to\infty}\frac{K(t,x,y)}{x},
\]
provided $l(t,\cdot)$ has properties needed
for the conserved quantity.  A convenient common control is
$\phi_0(x,y)=C(1+x)(1+y)$.  Deterministic initial conditions must still
satisfy every part of Assumption~\ref{assu:initial}; in particular, weak
convergence alone does not replace its moment and compact-containment
requirements.

\begin{remark}[The multiplicative kernel]\label{rmk:multiplicative-kernel}
Assumption~\ref{assu:kernel-5} excludes the classical multiplicative
kernel
\[
K(x,y)=xy
\]
of the Marcus--Lushnikov process.  The tail-smallness conditions in that
assumption make coagulations between two clusters whose types both escape
to infinity negligible.  This fails for the multiplicative kernel, for
which the proof would otherwise have to record interactions between two
asymptotically macroscopic clusters.

The decomposition used in the proof of
Proposition~\ref{prop:main continuous proposition} suggests an alternative
treatment for kernels of multiplicative order by restricting the test
functions to satisfy
\[
f(t,\infty,\infty,\infty)=0,
\qquad t\in[0,T].
\]
This suppresses the two-large-cluster contribution at the corresponding
boundary point.  It would also require upper-bound tests satisfying
$\bar g(t,\infty)=0$ and admissible lower-bound tilts satisfying
$\eta(t,\infty,\infty,\infty)=1$.  A separate reformulation of the
continuity estimate and the rate functions would be needed, and we do not
claim such an extension here.
\end{remark}

In this classical setting, the multi-type Flory equation reduces to the
usual Flory equation for mass loss due to gelation. For any
$g\in\mathscr C_c([0,\infty))$,
    \begin{equation}\label{eq:classic Flory equation}
        \begin{aligned}
            \langle g,\mu_t\rangle
            =&\langle g,\mu_0\rangle
            +\frac{1}{2}\int_0^t\!\int_{[0,\infty)^2}
            K(s,x,y)[g(x+y)-g(x)-g(y)]
            \mu_s(\diff x)\mu_s(\diff y)\diff s\\
            &-\int_0^t\!\langle g(\cdot)l(s,\cdot),\mu_s\rangle
            \langle x,\mu_0-\mu_s\rangle\diff s.
        \end{aligned}
    \end{equation}
This formulation agrees with the models studied in \cite{Fournier2004} and \cite{FOURNIER2009}, where the correction term involving $l(s,x)$ accounts for interactions between finite clusters and the gel phase.

Our large deviation results yield rate functions specialized to this
setting.  Write
\[
\Gamma_T(\pi,g):=\langle g_T,\pi_T\rangle-\langle g_0,\pi_0\rangle
-\int_0^T\langle\partial_sg_s,\pi_s\rangle\diff s.
\]
The upper rate becomes
\begin{multline*}
\mathcal{R}^K_{\rm upper}(\pi)
=\sup_{g \in \mathscr C_{\bar E}^{1,0}}\bigg\{
\Gamma_T(\pi,g)\\
-\frac{1}{2}\int^T_0\!\int_{[0,\infty)^2}
\left(e^{g_s(x+y)-g_s(x)-g_s(y)}-1\right)K(s,x,y)
\pi_s(\diff x)\pi_s(\diff y)\diff s\\
+\int_0^T\!\int_{[0,\infty)}(1-e^{-g_s(x)})l(s,x)
\langle x,\pi_0-\pi_s\rangle\pi_s(\diff x)\diff s\bigg\}.
\end{multline*}
For uniqueness-admissible tilts, the cost appearing in the local lower
bound takes the following form, where
$h_\eta:=\eta\log\eta-\eta+1$:
    \[
    \begin{aligned}
        \mathcal{R}^K_{\rm{lower,uniq}}(\pi):=
        \inf_{\eta}\bigg\{&\frac{1}{2} \int_0^T\int_{[0,\infty)^2}
        h_\eta(s,x,y)K(s,x,y)\\
        &\hspace{38mm}\times\pi_s(\diff x)\pi_s(\diff y)\diff s\\
        &+\int_0^T
        \langle h_\eta(s,\infty,\cdot)l(s,\cdot),\pi_s\rangle\\
        &\hspace{38mm}\times\langle x,\pi_0-\pi_s\rangle\diff s\bigg\}.
    \end{aligned}
    \]
Here the infimum is restricted to strictly positive symmetric tilts
$\eta \in \mathscr{C}_b([0,T] \times [0,\infty]^2)$ for which
$\pi$ is the unique solution of the weak form of
\eqref{eq:classic Flory equation} with perturbed kernel $\hat{\eta}K$,
starting from $\nu$.

Thus the upper bound applies globally, while the stated lower bound applies
locally around paths satisfying this tilted uniqueness condition.


\section{Hydrodynamic limit}\label{section 3}
This section proves tightness of the time-inhomogeneous coagulation processes
and identifies their subsequential limits.  We first introduce a closed set
$\mathscr{D}_T^a$ encoding uniform-in-time $\Phi$-moment bounds. We then
establish a stability estimate for the map
$\pi\longmapsto\left\langle f,\Lambda_\pi^{K,t}\right\rangle$ along uniformly
weakly convergent sequences satisfying the required initial-data
compatibility condition.
These estimates identify every limit point with a solution of the multi-type
Flory equation and will also be used in Section~\ref{section 4}.

\subsection{Technical lemmas}

Set $\Phi:=1+\phi_0+\phi^*$.  By
Assumptions~\ref{assu:initial-1} and \ref{assu:initial-3}, and by the
convergence $\mu_0^N\to\nu$, there exist constants $a>0$ and $N_0$
such that
\begin{equation}\label{eq:a}
\sup_{N>N_0}\int_{E\times E}\Phi(x,y)\mu_0^N(\diff x)\mu_0^N(\diff y)\leq a.
\end{equation}
For the constant $a$ determined above, we define the subset $\mathscr{D}_T^a\subset\mathscr{D}_T(\mathscr{M}_f^+(E))$ as:
\[
\begin{aligned}
    \mathscr{D}_T^a = \Bigg\{\mu\in \mathscr{D}_T(\mathscr{M}_f^+(E)) \bigg|& \sup_{t\in[0,T]}\croc{\Phi,\mu_t\otimes\mu_t}\leq a,\\
    &\sup_{t\in[0,T]}\croc{\Phi,\mu_0\otimes\mu_t}\leq a\Bigg\}.
\end{aligned}
\]
Recall that $\Phi$ is a nonnegative doubly sub-conserved quantity
(Definition~\ref{def:Conserved or sub-conserved quantities}).

\begin{lemma}\label{lemmadta}
    Under Assumption~\ref{assu:kernel} and  Assumption~\ref{assu:initial}, for any $N>N_0$, $\mu^N\in \mathscr{D}_T^a$ almost surely.
    Moreover, $\mathscr D_T^a$ is closed in
    $\mathscr{D}_T(\mathscr M_f^+(E))$ equipped with the $J_1$ topology. In
    particular, if $\mu$ is a limit of any convergent subsequence of $\mu^N$
    in $\mathscr{D}_T(\mathscr{M}_f^+(E))$, then
    $\mu\in\mathscr{D}_T^a$.
\end{lemma}
\begin{proof}\label{lemma:mu in DTa}

For any doubly sub-conserved quantity $g$, we have for any $0\leq r\leq s$,
    \[
    \int_{E} g(x,y) \mu_s^N(\diff x)\leq \int_{E} g(x,y) \mu_r^N(\diff x).
    \]
    Since $\Phi$ is doubly sub-conserved, it follows from the equation above and Fubini's theorem that, 
\[
\sup_{s\in[0,T]}\int_{E\times E}\Phi(x,y)\mu_s^N(\diff x)\mu_s^N(\diff y)\le
\int_{E\times E}\Phi(x,y)\mu_0^N(\diff x)\mu_0^N(\diff y),
\]
and
\[
\sup_{s\in[0,T]}\int_{E\times E}\Phi(x,y)\mu_0^N(\diff x)\mu_s^N(\diff y)\le
\int_{E\times E}\Phi(x,y)\mu_0^N(\diff x)\mu_0^N(\diff y).
\]
Given the upper bound~\eqref{eq:a},  for any $N>N_0$, the process $\mu^N$ belongs to the set $\mathscr{D}_T^a$ almost surely.

For any convergent sequence $(\sigma^i)\subset \mathscr{D}_T^a$ with
limit $\sigma$ and any $t\in[0,T]$, the time changes in the definition of
the $J_1$ topology provide $\sigma_{0}^i\to\sigma_0$ weakly and times $t_i\to t$ such that
$\sigma_{t_i}^i\to\sigma_t$ weakly. Thus,
    \[
    \langle \Phi,\sigma_t\otimes \sigma_t \rangle\leq \liminf_i\langle \Phi,\sigma_{t_i}^i\otimes \sigma_{t_i}^i \rangle\leq a.
    \]
    and
    \[
    \langle \Phi,\sigma_0\otimes \sigma_t \rangle\leq \liminf_i\langle \Phi,\sigma_0^i\otimes \sigma_{t_i}^i \rangle\leq a.
    \]
\end{proof}

\begin{prop}\label{prop:main continuous proposition}
    Suppose the coagulation kernel $K$ satisfies Assumption~\ref{assu:kernel} for some functions $\phi,\phi_0$. Let $f(t,x,y,z)\in \mathscr{C}_b(X_T)$ be fixed. Consider an element $\sigma\in \mathscr{D}_T^a$ with an initial measure $\sigma_0$ such that, for any fixed $t\in[0,T]$,
     the mapping $  y\mapsto\int_E\phi(t,x,y)\sigma_0(\diff x)$ is continuous on $E$. Then, the mapping
    \[
    \pi\in \mathscr{D}_T^a \mapsto
    \left(\croc{f,\Lambda_\pi^{K,t}}\right)_{t\in[0,T]}
    \in\mathscr{C}([0,T],\R)
    \]
    is continuous at $\sigma$ in the following sense: 
    for any sequence $(\sigma_\cdot^i)_i$ in $\mathscr{D}_T^a$ such that, for every $h\in \mathscr{C}_b(E)$, 
    \begin{equation}\label{eq:convergence in uniform topology in main prop}
        \lim_{i\to\infty}\sup_{t\in[0,T]}|\croc{h,\sigma_t^i}-\croc{h,\sigma_t}|=0,
    \end{equation}
    and for every compact set $C'\subset E$ and $t\in[0,T]$,
    \[
    \lim_{i\to\infty}\sup_{y\in C'}\left|\int_E\phi(t,x,y)\sigma_0^i(\diff x)-\int_E\phi(t,x,y)\sigma_0(\diff x)\right|=0,
    \]
   it follows that
    \[
    \lim_{i\to\infty} \sup_{t\in[0,T]}\left|\croc{f,\Lambda_{\sigma^i}^{K,t}}-\croc{f,\Lambda_\sigma^{K,t}}\right|=0.
    \]
    
\end{prop}

\begin{proof}[Proof of Proposition~\ref{prop:main continuous proposition}]

Fix $\varepsilon>0$.
    By Assumption~\ref{assu:kernel-5}, there exists $k\in \N$ such that
    \begin{equation}\label{eq:K/phi0<epsilon}
    \sup_{t\in[0,T], x\in C_k^c,y\in C_k^c} \frac{ \bar{K}(t,x,y)}{\phi_0(x,y)}<\varepsilon,
    \end{equation}
    and
    \[
    \sup_{t\in[0,T],x\in C_k^c,y\in E} \frac{ \phi(t,y,x)}{\phi_0(y,x)}<\varepsilon,
    \]
    where $(C_k)_{k\in \N}$ is an increasing collection of compact sets in $E$.
    Since the space $X_T$ is compact,  the function  $f\in\mathscr{C}_b(X_T)$ is uniformly continuous. Combined with Assumption~\ref{assu:kernel-3}, we can choose a compact set $C_l\supseteq C_k$ such that
    \[
    \sup_{t\in[0,T], x,z\in C_l^c,y\in C_k}|f(t,x,y,z)-f(t,\infty,y,\infty)|<\varepsilon,
    \]
   and 
    \[
    \sup_{t\in[0,T],x\in C_l^c,y\in C_k}\frac{ |\bar{K}(t,x,y)-\phi(t,x,y)|}{\phi_0(x,y)}<\varepsilon.
    \]
   For any set $D\subset E$, we introduce  the preimage set of $D$ in $E\times E$ as
    \[
    \textrm{pre}(D):=\{(x,y)\in E\times E:
    K(t,x,y,D)>0\text{ for some }t\in[0,T]\}.
    \]
   By Assumption~\ref{assu:kernel-4}, 
    there exists a compact set $C'\subset E\times E$ such that $C'\supseteq C_l\times C_l$ and for any $t\in[0,T]$, if $K(t,x,y,C_l)>0$ then $(x,y)\in C'$.
    Clearly, $(C')^c\cap \textrm{pre}(C_l)=\emptyset$.

    Let $\psi_{\varepsilon}:E\times E\mapsto [0,1]$ be a non-negative continuous symmetric function with compact support satisfying $\psi_{\varepsilon}(x,y)=1$ on $C'$, and let $\zeta_{\varepsilon}:E\mapsto [0,1]$ be a non-negative continuous function with compact support such that $\zeta_{\varepsilon}(y)=1$ on $C_k$. Set
\[
f^K_t(x,y):=\int_E f(t,x,y,z)K(t,x,y,\diff z),
\]
\[f_t^{\phi,1}(x,y):=f(t,x,\infty,\infty)\phi(t,y,x),\]
\[f_t^{\phi,2}(x,y):=f(t,\infty,y,\infty)\phi(t,x,y).\]
   
Combining these notations with the definition of $\Lambda_\pi^{K,t}$~\eqref{eq:equivalent to main continuous proposition}, we  decompose $$\croc{f,\Lambda_\pi^{K,t}}=\sum_{p=1}^5 I_{p,t}(\pi),$$
 where
 \begin{equation*}
            I_{1,t}(\pi)= \int_0^t
            \bigg(\frac{1}{2}\left\langle f^K_s\psi_{\varepsilon},\pi_s\otimes \pi_s\right\rangle 
            -\langle f_s^{\phi,2} \psi_{\varepsilon}, \pi_s\otimes \pi_s\rangle\bigg)\diff s,
    \end{equation*}

    \begin{equation*}
    \begin{aligned}
            I_{2,t}(\pi)=& \int_0^t
            \frac{1}{2}\bigg(\left\langle (f^K_s-f_s^{\phi,2})(1-\psi_{\varepsilon}) (1-\mathbbm{1}_{C_k^c\times C_k^c}),\pi_s\otimes \pi_s\right\rangle \\
            &\qquad-\langle f_s^{\phi,1} (1-\psi_{\varepsilon}) (1-\mathbbm{1}_{C_k^c\times C_k^c}), \pi_s\otimes \pi_s\rangle\bigg)\diff s,
            \end{aligned}
    \end{equation*}

    \begin{equation*}
            I_{3,t}(\pi)= \int_0^t
            \bigg(\frac{1}{2}\left\langle f^K_s(1-\psi_{\varepsilon}) \mathbbm{1}_{C_k^c\times C_k^c},\pi_s\otimes \pi_s\right\rangle 
            -\langle f_s^{\phi,2} (1-\psi_{\varepsilon}) \mathbbm{1}_{C_k^c\times C_k^c}, \pi_s\otimes \pi_s\rangle\bigg)\diff s,
    \end{equation*}

 \begin{equation*}
            I_{4,t}(\pi)=\int_0^t \int_{E\times E} f_s^{\phi,2}(x,y)\zeta_{\varepsilon}(y)  \pi_0(\diff x) \pi_s(\diff y)\diff s ,
    \end{equation*}
 \begin{equation*}
        \begin{aligned}
            I_{5,t}(\pi)=\int_0^t \int_{E\times E} f_s^{\phi,2}(x,y)\left(1-\zeta_{\varepsilon}(y)\right)  \pi_0(\diff x) \pi_s(\diff y)\diff s ,
        \end{aligned}
    \end{equation*}
The symmetrized loss term in $I_{2,t}$ follows by interchanging the
dummy variables $x$ and $y$ in one half of the integral; no symmetry
of $f$ is required.

We now analyse the difference $\Delta_p:=\sup_{0\le t\le T}|I_{p,t}(\sigma^i)-I_{p,t}(\sigma)|$ for each $p$.

\emph{Step 1: Analysis of $\Delta_1$.}
As $\sigma_t^i\to\sigma_t$ weakly, it follows that
    $\sigma_t^i\otimes\sigma_t^i\to\sigma_t\otimes\sigma_t$ weakly.
    Since  $\psi_{\varepsilon}$ has a compact support, say $C^*$, and 
    by Assumption~\ref{assu:kernel-4}, there exists a compact set
    $\tilde{C}$ such that $C^*\cap\operatorname{pre}(\tilde{C}^c)=\emptyset$,
    the functions $f_t^K\psi_\varepsilon$ and $f_t^{\phi,2}\psi_\varepsilon$
    are continuous and bounded on $E\times E$. Hence, we have
        \[
        \left\langle f_t^K\psi_\varepsilon, \sigma_t^i\otimes \sigma_t^i\right\rangle\to \left\langle f_t^K\psi_\varepsilon, \sigma_t\otimes \sigma_t\right\rangle
        \]
        and
        \[
        \left\langle f_t^{\phi,2}\psi_\varepsilon, \sigma_t^i\otimes \sigma_t^i\right\rangle\to \left\langle f_t^{\phi,2}\psi_\varepsilon, \sigma_t\otimes \sigma_t\right\rangle
        \]
    Moreover, using the bound $\bar{K}\le\phi_0$, on $\mathscr{D}_T^a$ we have
    \[
    \bigl|\bigl\langle f_t^K\psi_\varepsilon,\sigma_t^i\otimes\sigma_t^i\bigr\rangle\bigr|
    \le\|f\|_{\infty}\langle\phi_0,\sigma_t^i\otimes\sigma_t^i\rangle\le\|f\|_{\infty}a,
    \]
    and similarly
    \[
    \bigl|\bigl\langle f_t^{\phi,2}\psi_\varepsilon,\sigma_t^i\otimes\sigma_t^i\bigr\rangle\bigr|\le\|f\|_{\infty}\langle\phi_0,\sigma_t^i\otimes\sigma_t^i\rangle
    \le\|f\|_{\infty}a .
    \]
    These bounds hold for every $t\in[0,T]$, so the dominated convergence theorem
    yields
    \[\lim_{i\to\infty}\sup_{t\in[0,T]} \Big|I_{1,t}(\sigma^i)-I_{1,t}(\sigma)\Big|=0.\]

    \emph{Step 2: Analysis of $\Delta_2$.}
Since $1-\psi_{\varepsilon}=0$ on $C'$, $(C')^c\cap \textrm{pre}(C_l)=\emptyset$, $C_k\subseteq C_l$ and 
    $C_l\times C_l\subset C'$, the integrand in $I_{2,t}$ is non‑zero only for
    $(x,y)\in(C_l^c\times C_k)\cup(C_k\times C_l^c)$.
    For any $\mu\in\mathscr{D}_T^a$, we estimate the integral using the uniform continuity of $f$ and the proximity of $\bar{K}$ to $\phi$ in Assumption~\ref{assu:kernel-3}:
    \[
    \begin{aligned}
    &\int_0^T\!\Bigl|
    \bigl\langle(f^K_s-f_s^{\phi,2})(1-\psi_{\varepsilon})
    \mathbbm{1}_{C_l^c\times C_k},\mu_s\otimes\mu_s\bigr\rangle\Bigr|\diff s \\
    &\le\int_0^T\!\!\int_{C_l^c\times C_k}
    \Bigl[\int_E|f(s,x,y,z)-f(s,\infty,y,\infty)|
    \,K(s,x,y,\diff z) \\
    &\qquad+|f(s,\infty,y,\infty)|
    \,\bigl|\bar{K}(s,x,y)-\phi(s,x,y)\bigr|\Bigr]
    \, (1-\psi_{\varepsilon})(x,y)\mu_s(\diff x)\mu_s(\diff y)\diff s \\
    &\le\varepsilon\int_0^T\!\Bigl[
    \int_{E\times E}\phi_0(x,y)\,\mu_s(\diff x)\mu_s(\diff y)
    +\|f\|_{\infty}\int_{E\times E}\phi_0(x,y)\,\mu_s(\diff x)\mu_s(\diff y)
    \Bigr]\diff s \\
    &\le\varepsilon T(\|f\|_{\infty}+1)a .
    \end{aligned}
    \]
    Similarly,
    \begin{multline*}
    \int_0^T\!\bigl|\bigl\langle f_s^{\phi,1}(1-\psi_{\varepsilon})
    \mathbbm{1}_{C_l^c\times C_k},\mu_s\otimes\mu_s\bigr\rangle\bigr|\diff s\\
    \le \varepsilon \|f\|_{\infty} \int_0^T \int_{E\times E}
    \phi_0(y,x)\mathbbm{1}_{C_l^c}(x)
    \mu_s(\diff x) \mu_s(\diff y)\diff s
    \le\varepsilon T\|f\|_{\infty}a .
    \end{multline*}
    The same estimates hold for the symmetric part $(C_k\times C_l^c)$.
   By summing up all terms, we have
$$\lim_{i\to\infty}\sup_{t\in[0,T]} \Big|I_{2,t}(\sigma^i)-I_{2,t}(\sigma)\Big| \leq 2\varepsilon T(2\|f\|_{\infty}+1)a.$$

   \emph{Step 3: Analysis of $\Delta_3$.}
For any $\mu\in\mathscr{D}_T^a$, we utilize the tail control from Assumption~\ref{assu:kernel-5}
    \[
    \begin{aligned}
    \int_0^T\!\bigl|\bigl\langle
    f^K_s(1-\psi_{\varepsilon})\mathbbm{1}_{(C_k^c)^2},
    \mu_s\otimes\mu_s\bigr\rangle\bigr|\diff s
    &\le\|f\|_{\infty}\int_0^T\!\int_{(C_k^c)^2}
    \bar{K}(s,x,y)\,\mu_s(\diff x)\mu_s(\diff y)\diff s \\
    &\le\varepsilon\|f\|_{\infty}
    \int_0^T\!\int_{(C_k^c)^2}\phi_0(x,y)
    \,\mu_s(\diff x)\mu_s(\diff y)\diff s \\
    &\le\varepsilon T\|f\|_{\infty}a .
    \end{aligned}
    \]
Similarly,
\[
\begin{aligned}
&\int_0^T\!\Big|\left\langle
f_s^{\phi,2}(1-\psi_{\varepsilon})
\mathbbm{1}_{(C_k^c)^2},
\mu_s\otimes \mu_s\right\rangle\Big|\diff s\\
&\hspace{35mm}\leq \varepsilon T \|f\|_{\infty} a.
\end{aligned}
\]
These yield:
        $$\lim_{i\to\infty}\sup_{t\in[0,T]} \Big|I_{3,t}(\sigma^i)-I_{3,t}(\sigma)\Big|\leq 3\varepsilon T\|f\|_{\infty}a.$$

        \emph{Step 4: Analysis of $\Delta_4$.}
We split the difference into two terms:
    \[
    \begin{aligned}
    &\Bigl|\int_E f(t,\infty,y,\infty)\zeta_{\varepsilon}(y)
    \Bigl(\int_E\phi(t,x,y)\,\sigma_0^i(\diff x)\Bigr)\sigma_t^i(\diff y) \\
    &\qquad-\int_E f(t,\infty,y,\infty)\zeta_{\varepsilon}(y)
    \Bigl(\int_E\phi(t,x,y)\,\sigma_0(\diff x)\Bigr)\sigma_t(\diff y)\Bigr| \\
    &\le\Bigl|\int_E f(t,\infty,y,\infty)\zeta_{\varepsilon}(y)
    \Bigl(\int_E\phi(t,x,y)\,(\sigma_0^i-\sigma_0)(\diff x)\Bigr)
    \sigma_t^i(\diff y)\Bigr| \\
    &\qquad+\Bigl|\int_E f(t,\infty,y,\infty)\zeta_{\varepsilon}(y)
    \int_E\phi(t,x,y)\,\sigma_0(\diff x)\,
    \bigl[\sigma_t^i(\diff y)-\sigma_t(\diff y)\bigr]\Bigr| .
    \end{aligned}
    \]
    The function $y\mapsto f(t,\infty,y,\infty)\zeta_{\varepsilon}(y)
    \int_E\phi(t,x,y)\sigma_0(\diff x)$ is continuous and bounded, so the second
    term tends to $0$ as $i\to\infty$ because $\sigma_t^i\to\sigma_t$ weakly.
    For the first term, note that $\zeta_{\varepsilon}$ has compact support.
    Given $\varepsilon'>0$ we can choose $i_0$ such that for all $i\ge i_0$,
    \[
    \sup_{y\in \textrm{supp}(\zeta_{\varepsilon})}\Bigl|\int_E\phi(t,x,y)\,\sigma_0^i(\diff x)
    -\int_E\phi(t,x,y)\,\sigma_0(\diff x)\Bigr|<\varepsilon'.
    \]
    Then
    \begin{multline*}
    \Bigl|\int_E f(t,\infty,y,\infty)\zeta_{\varepsilon}(y)
    \Bigl(\int_E\phi(t,x,y)\,(\sigma_0^i-\sigma_0)(\diff x)\Bigr)
    \sigma_t^i(\diff y)\Bigr|\\
    \le\varepsilon'\|f\|_{\infty}\int_E\sigma_t^i(\diff y)
    \le\varepsilon'\|f\|_{\infty}\sqrt{a}.
    \end{multline*}
    Letting $i\to\infty$ and then $\varepsilon'\to0$ shows pointwise
    convergence for each $t$.  Moreover, each of the two original
    integrands is bounded in absolute value by $\|f\|_\infty a$ after
    integration in $y$, by the defining bounds of $\mathscr D_T^a$ and
    $\phi\leq\phi_0$.  The dominated convergence theorem therefore gives
    $$\lim_{i\to\infty}\sup_{t\in[0,T]} \Big|I_{4,t}(\sigma^i)-I_{4,t}(\sigma)\Big|=0.$$
        
        \emph{Step 5: Analysis of $\Delta_5$.}
For any $\mu\in \mathscr{D}_T^a$, using the fact $1-\zeta_\eps$ vanishes on $C_k$, we apply the tail estimate for $\phi$ from Assumption~\ref{assu:kernel-5} and get
        \[
        \begin{aligned}
            &\int_0^T  \Big| \int_{E} f(t,\infty,y,\infty) (1-\zeta_{\varepsilon}(y))\Big(\int_E\phi(t,x,y)\mu_0(\diff x)\Big) \mu_t(\diff y) \Big|\diff t\\
            \leq& \|f\|_{\infty}\int_0^T  \int_{E\times E}\phi(t,x,y)\mathbbm{1}_{C_k^c}(y)\mu_0(\diff x) \mu_t(\diff y)\diff t\\
            \leq& \varepsilon\|f\|_{\infty}\int_0^T \Big|\int_{E\times E}\phi_0(x,y) \mu_0(\diff x) \mu_t(\diff y) \Big|\diff t\\
            \leq& \varepsilon T\|f\|_{\infty} a.
        \end{aligned}
        \]
Hence,
        $$\lim_{i\to\infty}\sup_{t\in[0,T]} \Big|I_{5,t}(\sigma^i)-I_{5,t}(\sigma)\Big|\leq 2\varepsilon T\|f\|_{\infty}a.$$

 Combining all steps and taking  $\varepsilon\to 0$, the proof is complete.
\end{proof}

\begin{lemma}\label{lemma: cross term to 0}
 Under Assumption~\ref{assu:kernel} and  Assumption~\ref{assu:initial},
for any $\delta>0$, there exists $N_1$ sufficiently large, such that for all $N>N_1$,
   \begin{equation}\label{eq:cross term to 0}\frac{1}{N} \int_0^T \int_{E} \bar{K}(t,x,x)\mu_t^N(\diff x)\diff t \le \delta,\qquad {a.s.}.
       \end{equation}
 
\end{lemma}

\begin{proof}[Proof of Lemma~\ref{lemma: cross term to 0}]
By Lemma~\ref{lemmadta}, for any $N>N_0$, the process $\mu^N$ belongs to the set $\mathscr{D}_T^a$ almost surely.
To bound the integral of the kernel along the diagonal, fix $\varepsilon>0$, for all $t\in[0,T]$, we partition $E$ into a compact set $C_k$ and its complement $C_k^c$ as in the proof of Proposition~\ref{prop:main continuous proposition}. On the compact set $[0,T]\times C_k$,  the kernel $K$ is upper bounded by some constant $M>0$ , we have,
        \[
        \sup_{0\le t\le T}\frac{1}{N}\int_{C_k} \bar{K}(t,x,x)\mu_t^N(\diff x)\leq \frac{M}{N}\sup_{0\le t\le T}\langle1,\mu_t^N\rangle\le \frac{M\sqrt{a}}{N}.
        \]
For any $\delta>0$, by choosing $N$ sufficiently large, we have
        \[
        \int_0^T \frac{1}{N} \int_{C_k} \bar{K}(t,x,x)\mu_t^N(\diff x)\diff t\leq\delta/2.
        \]

On the complement set $C_k^c$, by~\eqref{eq:K/phi0<epsilon}, we can bound the integrand by $\varepsilon\phi_0(x,x)$. Note that the diagonal term
\[\frac{1}{N} \int_{E} \phi_0(x,x)\mu_t^N(\diff x)\le \int_{E\times E}\phi_0(x,y)\mu_t^N(\diff x)\mu_t^N(\diff y).\]
On  the set $\mathscr{D}_T^a$, the last term is bounded by $a$. Thus, by choosing $\varepsilon\le \delta/(2aT)$,  we have 
        \[
        \int_0^T\frac{1}{N} \int_{C_k^c} \bar{K}(t,x,x)\mu_t^N(\diff x)\diff t\leq\delta/2.
        \]
Combining these two bounds, we conclude that \eqref{eq:cross term to 0} holds almost surely for N sufficiently large.
   
\end{proof}

\subsection{Proof of the hydrodynamic limit}

\begin{proof}[Proof of Theorem \ref{thm:LLN}]
We first give the compactness and martingale estimates; this also records
why the time-inhomogeneous extension of the argument in
\cite{AIMconvergence} is valid.  Lemma~\ref{lemmadta} implies compact
containment, since
\[
\mathcal K_a:=\{\rho:\langle\Phi,\rho\otimes\rho\rangle\leq a\}
\]
is compact in $\mathscr M_f^+(E)$.  Indeed, the term $1$ bounds total
mass, while the compact sublevels in Assumption~\ref{assu:initial-3}
make the product measures tight; their identical marginals are then
tight as well.

Fix $g\in\mathscr C_b(E)$.  The semimartingale decomposition is
\[
\langle g,\mu_t^N\rangle=\langle g,\mu_0^N\rangle
+A_t^{g,N}+M_t^{g,N},
\]
where
\[
A_t^{g,N}=\int_0^t L_s^{K,N}\langle g,\mu_s^N\rangle\diff s
\]
and $M^{g,N}$ is a square-integrable martingale.  Since
$|L[g](x,y,z)|\leq3\|g\|_\infty$ and
$\bar K\leq\phi_0\leq\Phi$, for every stopping time $\tau_N\leq T$
and $0\leq\theta\leq\delta$,
\begin{equation}\label{eq:LLN-Aldous-drift}
|A_{(\tau_N+\theta)\wedge T}^{g,N}-A_{\tau_N}^{g,N}|
\leq \frac32\,\|g\|_\infty a\delta.
\end{equation}
Moreover,
\begin{equation}\label{eq:LLN-mart-qv}
\mathbb E\!\left[
\langle M^{g,N}\rangle_T\right]
\leq \frac{9\|g\|_\infty^2aT}{2N},
\qquad
\mathbb E\!\left[\sup_{t\leq T}|M_t^{g,N}|^2\right]
\leq\frac{18\|g\|_\infty^2aT}{N}.
\end{equation}
The first estimate follows directly from the generator and the second
from the compensator and Doob's inequality.  Equations
\eqref{eq:LLN-Aldous-drift}--\eqref{eq:LLN-mart-qv}, first for a countable
convergence-determining family in $\mathscr C_b(E)$ and then together
with compact containment, verify the standard Aldous tightness criterion
on $\mathscr D_T(\mathscr M_f^+(E))$; see
\cite[Chapter~3, Section~8]{ethier}.  Hence $(\mu^N)$ is tight.

For every such $g$, a jump of $\langle g,\mu^N\rangle$ has absolute
size at most $3\|g\|_\infty/N$.  The preceding Aldous estimate therefore
gives $C$-tightness of all the scalar projections.  Compact containment
and the convergence-determining property imply that every limit path is
almost surely continuous.

Now take a weakly convergent subsequence.  By the Skorokhod
representation theorem we may realize it so that $\mu^N\to\mu$ almost
surely in the $J_1$ topology.  Lemma~\ref{lemmadta} gives
$\mu^N,\mu\in\mathscr D_T^a$.  Since $\mu$ is continuous, this convergence
is uniform after testing against every $h\in\mathscr C_b(E)$
\cite[p.~112]{billing}, and hence verifies
\eqref{eq:convergence in uniform topology in main prop}.  Also
$\mu_0=\nu$.  For every $t\in[0,T]$, the function
\[
\int_E \phi(t,x,y) \, \mu_0^N(\diff x)
\]
is a finite sum of continuous functions in $y$,  hence continuous. The uniform convergence condition in Assumption~\ref{assu:initial-2} guarantees the limit map
\[
y \mapsto \int_E \phi(t,x,y) \, \mu_0(\diff x)
\]
remains continuous on $E$. Thus all hypotheses of
Proposition~\ref{prop:main continuous proposition} are satisfied.
Consequently, for all $g\in\mathscr C_c(E)$, almost surely,
\[
\lim_{N \to \infty} \sup_{t \in [0,T]} \left| \langle {L[\bar{g}]}, \Lambda_{\mu^N}^{K,t} \rangle - \langle L[\bar{g}], \Lambda_{\mu}^{K,t} \rangle \right| = 0
\]
and
\[
\lim_{N \to \infty} \sup_{t \in [0,T]} \bigl| \bigl( \langle g, \mu_t^N \rangle - \langle g, \mu_0^N \rangle \bigr) - \bigl( \langle g, \mu_t \rangle - \langle g, \mu_0 \rangle \bigr) \bigr| = 0.
\]

The finite system conserves
$\langle\phi(s,\cdot,y),\mu_t^N\rangle$ for every fixed $s,y$.
Consequently, its gel part in $\Lambda_{\mu^N}^{K,t}$ is zero, and the
generator identity gives
\[
\sup_{t\leq T}\left|
\langle g,\mu_t^N-\mu_0^N\rangle
-\langle L[\bar g],\Lambda_{\mu^N}^{K,t}\rangle
-M_t^{g,N}\right|
\leq\frac{3\|g\|_\infty}{2N}
\int_0^T\!\int_E\bar K(s,x,x)\mu_s^N(\diff x)\diff s.
\]
The right-hand side tends to zero by
Lemma~\ref{lemma: cross term to 0}, while $M^{g,N}$ tends to zero in
$L^2$ by \eqref{eq:LLN-mart-qv}.  Passing to the limit with the preceding
uniform convergences yields
\[
\langle \mu_t, g \rangle - \langle \mu_0, g \rangle - \langle L[\bar{g}], \Lambda_{\mu}^{K,t} \rangle = 0 \qquad \text{for all } t \in [0,T] \text{ almost surely}.
\]
Taking a countable sup-norm dense family of compactly supported test
functions in $\mathscr C_0(E)$ and then using uniform approximation shows that the weak
identity holds simultaneously for every $g\in\mathscr C_c(E)$.  The
$\Phi$-bounds imply the two integrability requirements in
Definition~\ref{def:Flory}.  Finally, the finite systems conserve
$\langle\phi(s,\cdot,y),\mu_t^N\rangle$ for every fixed $s,y$; lower
semicontinuity and the Portmanteau theorem therefore give Item~4 of that
definition in the limit. Hence $\mu$ is a solution to the
time-inhomogeneous multi-type Flory equation~\eqref{lambda-Flory}.

\end{proof}

\section{Proofs of the large-deviation bounds}\label{section 4}

We introduce the exponential martingales associated with changing the
kernel from $K$ to an absolutely continuous kernel $K'$. To retain the
marks $(x,y,z)$ of an event, we work on the usual marked canonical
extension of the path space and project to the coordinate path at the end.

\subsection{The evolution equation driven by an orthogonal martingale measure}
The $\mathscr M_f^+(E)$-valued process $\mu^N$ has generator $L^{K,N}$
from Definition~\ref{def:process}.  We use the orthogonal-martingale-measure
construction of \cite[Section~4]{ldpsun} to write its semimartingale
decomposition on $E^2\times E$.  The following notation retains the marks of
individual coagulation events.

Let $\mathcal N^N(\diff s\diff x\diff y\diff z)$ be a marked counting
measure on the ordered space $E^2\times E$.  On the marked canonical
extension, attach an independent fair orientation to every coagulation
event: if the unordered pair $\{x,y\}$ produces $z$ at time $s$, place
one atom at either $(s,x,y,z)$ or $(s,y,x,z)$, each with probability
$1/2$. (When $x=y$, the two choices coincide.)  This convention neither
duplicates nor fractionally splits an event, and its compensator is
$N\kappa^{N,K}[\mu^N]$, where for
$\xi\in\mathscr D_T(\mathscr M_N^+(E))$,
\[\kappa^{N,K}[\xi]((0,t]\times U):=\frac{1}{2}\int_0^t \int_U K(s,x,y,\diff z)\xi_s(\diff x)(\xi_s-\frac{\delta_x}{N})(\diff y) \diff s,\]
and $\kappa^{N,K}[\xi] \equiv 0$ otherwise.

Define the compensated orthogonal martingale measure
\[
\widetilde{\mathcal N}^{N,K}:=
\mathcal N^N-N\kappa^{N,K}[\mu^N].
\]
For Borel sets $U_1,U_2\subset E^2\times E$, its predictable covariance is
\[
\croc{\widetilde{\mathcal N}^{N,K}_t(U_1),
\widetilde{\mathcal N}^{N,K}_t(U_2)}
=N\kappa^{N,K}[\mu^N]((0,t]\times(U_1\cap U_2)).
\]
In this framework, $\mu^N$ satisfies the following evolution equation:
\[
\begin{aligned}
\mu_t^N=&\mu_0^N+\frac{1}{N}\int_0^t\int_{E^2\times E} (\delta_{z} - \delta_{x} -\delta_{y}) \mathcal N^N(\diff s\diff x\diff y\diff z)\\
=&\mu_0^N+\frac{1}{N}\int_0^t\int_{E^2\times E} (\delta_{z} - \delta_{x} -\delta_{y}) \widetilde{\mathcal N}^{N,K}(\diff s\diff x\diff y\diff z)\\
&+ \int_0^t\int_{E^2\times E} (\delta_{z} - \delta_{x} -\delta_{y}) \kappa^{N,K}[\mu^N](\diff s\diff x\diff y\diff z)
\end{aligned}
\]
Also, for any test function $g\in \mathscr{C}_b^{1,0}([0,T]\times E)$,
write $\diff\mathbf u=\diff s\diff x\diff y\diff z$.  Then
\begin{equation}\label{eq:mart problem g}
    \begin{aligned}
        \croc{g_t,\mu_t^N} &- \croc{g_0,\mu_0^N}
        -\int_0^t \croc{\partial_sg_s,\mu_s^N}\diff s\\
        &-\int_{(0,t]\times E^2\times E} L[g_s](x,y,z)
        \kappa^{N,K}[\mu^N](\diff\mathbf u)\\
        =& \frac{1}{N}\int_{(0,t]\times E^2\times E}L[g_s](x,y,z)\\
        &\hspace{22mm}\times
        \widetilde{\mathcal N}^{N,K}(\diff\mathbf u).
    \end{aligned}
\end{equation}
Clearly, the right-hand side of \eqref{eq:mart problem g} is a martingale.

\subsection{Coupling}

\begin{prop}[Well-posedness of the finite particle system]\label{prop:mart problem}
    For any kernel $K$ satisfying Assumption~\ref{assu:kernel} with $\phi$ and $\phi_0$ and any initial measure $\mu_0^N\in\mathscr{M}_N^+(E)$, there exists a unique probability measure $\mathbb{P}^{N,K}$ on
    $\mathscr{D}_T(\mathscr{M}_f^+(E))$
    such that for all $F \in \mathscr{C}_b(\mathscr{M}_f^+(E))$,
    \begin{equation}\label{eq:martingale problem F}
        F(\mu_t)-F(\mu_s)- \int_s^t L_r^{K,N}F(\mu_r)\diff r
    \end{equation}
    is an $(\mathscr F_t,\mathbb{P}^{N,K})$-martingale, where
    $\mathscr F_t:=\bigcap_{\varepsilon>0}\sigma\{\mu_u:
    0\leq u\leq (t+\varepsilon)\wedge T\}$ is the right-continuous
    natural filtration of the coordinate process. Moreover, for all
    $G\in \mathscr{C}_b^{1,0}([0,T]\times \mathscr{M}_f^+(E))$,
    \[
    G(t,\mu_t)-G(s,\mu_s)-\int_s^t \partial_r G(r,\mu_r) \diff r - \int_s^t L_r^{K,N}G_r(\mu_r)\diff r
    \]
    is also a martingale under the same filtration and probability.
\end{prop}

\begin{proof}
The total jump rate is finite at every state, because a state contains
only finitely many clusters and each $K(t,x,y,\cdot)$ is finite.  Every
jump reduces the number of clusters by one, so explosion is impossible.
The standard construction of a time-inhomogeneous pure-jump process then
gives existence and uniqueness; see \cite[Section~4.7.A]{ethier}.  The two
martingale identities follow from Dynkin's formula.
\end{proof}

\begin{defi}[Absolutely continuous kernels]\label{Def:perturbed}
    We say that $K'\ll K$ if, for every Borel set $V\subset E$,
    every $x,y\in E$, and almost every $t\in[0,T]$,
    \[K(t,x,y,V)=0\quad\Rightarrow\quad K'(t,x,y,V)=0.\]
    We write
    \[\eta^{(K'|K)}(t,x,y,z):=
      \frac{\diff K'(t,x,y,\cdot)}{\diff K(t,x,y,\cdot)}(z).\]
    We always choose a jointly measurable version of this density, which
    exists because $E$ is a standard Borel space.
    In particular, for any $f \in \mathscr C_b^{1,0}([0,T]\times E)$, we define a $f$-perturbed kernel $K^f$ by
    \[
    K^f(t,x,y,\diff z) = e^{f_t(z)-f_t(x)-f_t(y)} K(t,x,y,\diff z).
    \]
    Then the corresponding Radon-Nikodym derivative is given by
    \[
    \eta^{(K^f|K)}(t,x,y,z)=e^{L[f_t](x,y,z)}.
    \]
\end{defi}

\begin{defi}[Exponential martingales]
    For any two kernels $K'\ll K$ with the Radon-Nikodym derivative $\eta^{(K'|K)}$, for any $\mu \in \mathscr{D}_T(\mathscr{M}_f^+(E))$, define
    \begin{equation}
        \begin{aligned}
            \mathcal{M}_t^{N,(K'|K)} :=& \exp\left\{\int_0^t \int_{E\times E\times E} \log\left(\eta^{(K'|K)}(s,x,y,z)\right)\mathcal N^N(\diff s\diff x\diff y\diff z)\right.\\
            &\left. - N\int_0^t \int_{E\times E\times E}\left(\kappa^{N,K'}[\mu](\diff s\diff x\diff y\diff z)-\kappa^{N,K}[\mu](\diff s\diff x\diff y\diff z)\right)
            \right\}.
        \end{aligned}
    \end{equation}
    Especially, for the $f$-perturbed kernel $K^f = e^{L[f]}K$, we have
    \begin{equation}\label{eq:expmart f-perturbed}
        \begin{aligned}
            \mathcal{M}_t^{N,(K^f|K)} :=& \exp\left\{\int_0^t \int_{E\times E\times E} \left[f_s(z)-f_s(x)-f_s(y)\right]\mathcal N^N(\diff s\diff x\diff y\diff z) \right.\\
            &\left. -N \int_0^t \int_{E\times E\times E}\left(e^{f_s(z)-f_s(x)-f_s(y)}-1\right)\kappa^{N,K}[\mu](\diff s\diff x\diff y\diff z)
            \right\}.
        \end{aligned}
    \end{equation}
    For any $s,t\in [0,T]$, let
    \[\mathcal{M}_{s,t}^{N,(K'|K)} := \frac{\mathcal{M}^{N,(K'|K)}_t}{\mathcal{M}^{N,(K'|K)}_s}.\]
\end{defi}

\begin{lemma}[Change of measure]\label{lemma:coupling}
    For any two kernels $K$ and $K'$ satisfying Assumption~\ref{assu:kernel} for some $\phi,\phi_0$ and $\phi',\phi_0'$ respectively, any deterministic initial configuration $\mu_0^N\in \mathscr{M}_N^+(E)$ satisfying 
    \[
    \croc{1+\phi_0(x,y)+\phi_0'(x,y),\mu_0^N(\diff x)\mu_0^N(\diff y)}<\infty
    \]
    and suppose $K'\ll K$ with a bounded Radon--Nikodym derivative
    $\eta^{(K'|K)}$. Then the process
    \[
    \big((\mathcal{M}_t^{N,(K'|K)})_{t\geq 0},(\mathscr{F}_t)_{t\geq0}, \mathbb{P}^{N,K}\big)
    \]
    is a martingale. Moreover, $\mathbb{P}^{N,K'}$ is locally absolutely
    continuous with respect to $\mathbb{P}^{N,K}$ with density process
    $\mathcal{M}_t^{N,(K'|K)}$. If $\eta^{(K'|K)}$ is bounded away from
    zero, the two laws are locally equivalent. In particular,
    \begin{equation}
        \begin{aligned}
            \frac{\diff\mathbb{P}^{N,K^f}}{\diff\mathbb{P}^{N,K}}\Big|_{\mathcal{F}_t} =& \mathcal{M}_t^{N,(K^f|K)}\\
            =& \exp \left\{ N\left(\croc{f_t,\mu_t}-\croc{f_0,\mu_0}-\int^t_0 \croc{\partial_s f_s,\mu_s}\diff s\right)\right.\\
            &\left. -N \int^t_0\int_{E\times E\times E} \left(e^{f_s(z)-f_s(x)-f_s(y)}-1\right) \kappa^{N,K}[\mu](\diff s\diff x\diff y\diff z)\right\}.
        \end{aligned}
    \end{equation}
\end{lemma}

The proof is the standard Radon-Nikodym change-of-measure formula for marked point processes;
see also \cite[Section~4]{ldpsun}.

\subsection{Exponential tightness}

Fix a bounded complete metric $d$ inducing weak convergence on
$\mathscr M_f^+(E)$ and a countable convergence-determining family
$(f_j)_{j\geq1}\subset\mathscr C_b(E)$ with
$\|f_j\|_\infty\leq1$. On each compact subset of
$\mathscr M_f^+(E)$, these coordinates generate the weak topology; hence
finitely many of them control the metric uniformly to any prescribed
accuracy.

\begin{lemma}[Exponential tightness]\label{lemma:exp tight}
Under Assumptions~\ref{assu:kernel} and \ref{assu:initial}, the laws of
$\mu^N$ are exponentially tight in
$\mathscr D_T(\mathscr M_f^+(E))$.
\end{lemma}

\begin{proof}
We verify exponential compact containment and the exponential Aldous
condition.  First, we have proved in the proof of Theorem~\ref{thm:LLN} that the set
\[
\mathcal K_a:=\{\rho\in\mathscr M_f^+(E):
\langle\Phi,\rho\otimes\rho\rangle\leq a\}
\]
is compact. Lemma~\ref{lemmadta} therefore gives the deterministic
compact-containment bound
\[
\mathbb P^{N,K}(\mu_t^N\in\mathcal K_a\text{ for all }t\leq T)=1
\qquad(N>N_0).
\]

Let $f\in\mathscr C_b(E)$ with $\|f\|_\infty\leq1$, and let $\tau_N$ be a
stopping time bounded by $T$ (with the process stopped at $T$).
Let
\[
\sigma_N:=\inf\left\{u\in[0,\delta]:
\langle f,\mu_{(\tau_N+u)\wedge T}^N-\mu_{\tau_N}^N\rangle
>\varepsilon\right\}\wedge\delta.
\]
Apply optional sampling, conditionally on $\mathscr F_{\tau_N}$, to the
Radon-Nikodym density martingale on the random interval
$[\tau_N,(\tau_N+\sigma_N)\wedge T]$.  Since
$|L[f]|\leq3$ and $\bar K\leq\phi_0\leq\Phi$, this gives
\[
\begin{aligned}
&\mathbb P^{N,K}\left(
\sup_{0\leq u\leq\delta}
|\langle f,\mu_{(\tau_N+u)\wedge T}^N-\mu_{\tau_N}^N\rangle|
>\varepsilon\right)\\
&\qquad\leq
2\exp\left\{-N\lambda\varepsilon
+\frac{Na\delta}{2}(e^{3\lambda}+1)\right\}.
\end{aligned}
\]
The factor $2$ comes from applying the same stopped argument to $-f$.
Given $L>0$ and $\varepsilon>0$, choose $\lambda$ first and then
$\delta$ so that the right-hand side has logarithmic rate at most $-L$.

Since
$\mathcal K_a$ is compact and the coordinates $\rho\mapsto
\langle f_j,\rho\rangle$ generate its topology, for every
$\varepsilon>0$ there are $J<\infty$ and $r_1,\ldots,r_J>0$ such that
for $\rho,\rho'\in\mathcal K_a$,
\[
|\langle f_j,\rho-\rho'\rangle|\leq r_j\quad(1\leq j\leq J)
\quad\Longrightarrow\quad d(\rho,\rho')\leq\varepsilon.
\]
Applying the scalar estimate to these finitely many coordinates and using
a union bound yields
\[
\lim_{\delta\downarrow0}\limsup_{N\to\infty}\frac1N\log
\sup_{\tau_N}\mathbb P^{N,K}\left(
\sup_{0\leq u\leq\delta}
   d(\mu_{(\tau_N+u)\wedge T}^N,\mu_{\tau_N}^N)>\varepsilon\right)
=-\infty.
\]
This is the exponential Aldous condition.  The exponential-tightness
criterion of \cite[Theorem~4.1]{fengkurtz}, together with exponential
compact containment, proves the claim.
\end{proof}

\subsection{Upper bound}

Fix $f\in\mathscr C_{\bar E}^{1,0}$ and define
\[ J_T(\pi,f)=\croc{f_T,\pi_T} - \croc{f_0,\pi_0} - \int^T_0 \croc{\partial_sf_s,\pi_s}\diff s-\croc{{e^{L[\bar{f}]}-1},\Lambda_\pi^{K,T}}.\]

\begin{prop}\label{prop:upper another}
    Assume the conditions of Theorem~\ref{thm:LDP upper} hold.
    Let \(\pi \in \mathscr{D}_T(\mathscr{M}_f^+(E))\).

    \begin{enumerate}
        \item If \(\iota_{\nu}(\pi_0) + \mathcal{R}^K_{\mathrm{upper}}(\pi) < \infty\), then for every \(\delta>0\) there exists an open neighborhood \(V\) of \(\pi\) such that
              \[
                  \limsup_{N\to\infty} \frac{1}{N}\log \mathbb{P}^{N,K}(V)
                  \le -\bigl(\iota_{\nu}(\pi_0) + \mathcal{R}^K_{\mathrm{upper}}(\pi)\bigr) + \delta .
              \]
        \item If \(\iota_{\nu}(\pi_0) + \mathcal{R}^K_{\mathrm{upper}}(\pi) = \infty\), then for every \(M>0\) there exists an open neighborhood \(U\) of \(\pi\) such that
              \[
                  \limsup_{N\to\infty} \frac{1}{N}\log \mathbb{P}^{N,K}(U) \le -M .
              \]
    \end{enumerate}
\end{prop}
\begin{proof}

    Case 1: If $\pi\notin\mathscr C([0,T],\mathscr M_f^+(E))$, the
    closedness of the continuous paths in the $J_1$ topology provides an
    open neighborhood $V$ of $\pi$ separated from paths whose largest
    jump is sufficiently small.  On the compact set $\mathcal K_a$, each
    jump of every coordinate $\langle f_j,\mu^N\rangle$ is at most
    $3/N$.  Compactness and the convergence-determining property imply
    that the largest $d$-jump tends to zero uniformly.  Hence
    $\mathbb P^{N,K}(V)=0$ for all sufficiently large $N$.

    Case 2: If $\pi\notin \mathscr{D}_T^a$, we can take an open neighborhood $V$ of $\pi$ such that $V\cap \mathscr{D}_T^a=\emptyset$. Consequently, $\mathbb{P}^{N,K}(V)=0$ for all $N>N_0$, which implies the conclusion.

    Case 3: If $\pi_0\neq\nu$, take a neighborhood whose initial values
    lie in the $d$-ball centered at $\pi_0$ with radius
    $d(\pi_0,\nu)/2$.  Since $d(\mu_0^N,\nu)\to0$, this neighborhood has
    probability zero for all sufficiently large $N$.

    Case 4: Suppose that $\pi\notin\mathscr A_\phi$.  We claim that an
    open neighborhood of $\pi$ contains no finite-particle path for all
    sufficiently large $N$.  Otherwise there would exist $N_k\uparrow
    \infty$ and paths $\rho^k$, reachable from $\mu_0^{N_k}$, such that
    $\rho^k\to\pi$ in the $J_1$ topology.  Cases 1--3 have already been
    excluded, so $\pi$ is continuous, belongs to $\mathscr D_T^a$, and
    starts from $\nu$.  Hence $\rho_t^k\to\pi_t$ weakly for every $t$.
    Conservation in the finite system and Portmanteau's theorem give,
    for every $t,s\in[0,T]$ and $y\in E$,
    \begin{align*}
    \int_E\phi(s,x,y)\pi_t(\diff x)
    &\leq\liminf_{k\to\infty}
      \int_E\phi(s,x,y)\rho_t^k(\diff x)\\
    &=\lim_{k\to\infty}
      \int_E\phi(s,x,y)\mu_0^{N_k}(\diff x)
      =\int_E\phi(s,x,y)\nu(\diff x).
    \end{align*}
    The last equality is Assumption~\ref{assu:initial-2}(c), applied to
    the compact singleton $\{y\}$.  Since $\pi\in\mathscr D_T^a$ and
    $\bar K,\phi\leq\phi_0$, the two absolute-integrability conditions in
    the definition of $\mathscr A_\phi$ also hold.  Thus $G_\pi$ is
    nonnegative and $\pi\in\mathscr A_\phi$, a contradiction.  The claimed
    neighborhood therefore has probability zero for all large $N$.

    Case 5:
  If $\iota_{\nu}(\pi_0)+\mathcal{R}^K_{\rm{upper}}(\pi)<\infty$, then by \eqref{eq:upper-rate-compact}, we can find $f\in\mathscr C_{\bar E}^{1,0}$ such that
    \[
    J_T(\pi,f)>\mathcal{R}^K_{\rm{upper}}(\pi)-\frac{\delta}{3}.
    \]
Consider the $f$-perturbed martingale $\mathcal{M}_T^{N,(K^f|K)}$.
Lemma~\ref{lemma: cross term to 0} and conservation of $\phi$ imply that,
for all sufficiently large $N$ and any measurable $V$,
 \[
    \begin{aligned}
        1=& \mathbb{E}^{N,K} \mathcal{M}_T^{N,(K^f|K)}\\
        =& \mathbb{E}^{N,K} \exp\Bigg\{N\Big[J_T(\mu,f)-J_T(\pi,f)+J_T(\pi,f)\\
        &\qquad +\frac{1}{2N} \int_0^T \int_{E\times E} \left(e^{f_t(z)-2f_t(x)}-1\right)K(t,x,x,\diff z) \mu_t(\diff x) \diff t\Big]\Bigg\}\\
        \geq& \mathbb{E}^{N,K}\Big[ \exp\{-N|J_T(\mu,f)-J_T(\pi,f)|+ N J_T(\pi,f)-\frac{N\delta}{3}\} \ind{\mu\in V}\Big].
    \end{aligned}
    \]

    We next justify the required local estimate without asserting ordinary
    continuity on all of $\mathscr D_T^a$.  If no such neighborhood and
    $N'$ existed, we could choose $N_k\uparrow\infty$ and reachable paths
    $\rho^k\to\pi$ for which
    $|J_T(\rho^k,f)-J_T(\pi,f)|>\delta/3$.  Since $\pi$ is continuous,
    the $J_1$ convergence is uniform after testing against bounded
    continuous functions.  Proposition~\ref{prop:main continuous proposition}
    and Assumption~\ref{assu:initial-2}(c) then imply
    $J_T(\rho^k,f)\to J_T(\pi,f)$, a contradiction.  Thus there are an
    open neighborhood $V=V(\delta,f)$ of $\pi$ and $N'\in\mathbb N$ such
    that every reachable path $\mu\in V$ with $N>N'$ satisfies
    \[
    |J_T(\mu,f)-J_T(\pi,f)|\leq \frac{\delta}{3}.
    \]
  Therefore,
    \[
    \limsup_{N\to\infty} \frac{1}{N} \log \mathbb{P}^{N,K}(V)\leq -J_T(\pi,f)+\frac{2\delta}{3}\leq -(\iota_{\nu}(\pi_0) + \mathcal{R}^K_{\rm{upper}}(\pi))+\delta.
    \]

    If $\mathcal{R}^K_{\rm{upper}}(\pi)=\infty$, choose
    $f\in\mathscr C_{\bar E}^{1,0}$ with $J_T(\pi,f)\geq M+1$.
    Repeating the argument gives an open neighborhood $U$ such that
    \[
    \limsup_{N\to\infty}\frac{1}{N} \log \mathbb{P}^{N,K}(U)\leq -M.
    \]
\end{proof}

\begin{proof}[Proof of Theorem~\ref{thm:LDP upper}]
    Let \(C \subset \mathscr{D}_T(\mathscr{M}_f^+(E))\) be compact. 
    Write $I(\pi)=\iota_\nu(\pi_0)+
    \mathcal R^K_{\mathrm{upper}}(\pi)$.  For each $\pi\in C$,
    Proposition~\ref{prop:upper another} provides an open neighborhood
    $W_\pi$ satisfying
    \[
        \limsup_{N\to\infty} \frac{1}{N}\log \mathbb{P}^{N,K}(W_\pi)
        \le 
        \begin{cases}
            -I(\pi)+\delta, & \text{if }I(\pi)<\infty,\\[4pt]
            -M, & \text{if }I(\pi)=\infty,
        \end{cases}
    \]
    for any prescribed \(\delta, M > 0\).

    The collection \(\{W_\pi : \pi \in C\}\) is an open cover of \(C\). By compactness, there exists a finite subcover
    \[
        C \subset \bigcup_{i=1}^{k} V_{\pi^{(i)}} \;\cup\; \bigcup_{j=1}^{l} U_{\tilde{\pi}^{(j)}},
    \]
    where each \(V_{\pi^{(i)}}\) (resp. \(U_{\tilde{\pi}^{(j)}}\)) is a neighbourhood of a point \(\pi^{(i)}\) with finite (resp. infinite) rate \(\iota_{\nu}(\pi^{(i)}_0) + \mathcal{R}^K_{\mathrm{upper}}(\pi^{(i)})\) (resp. \(\iota_{\nu}(\tilde{\pi}^{(j)}_0) + \mathcal{R}^K_{\mathrm{upper}}(\tilde{\pi}^{(j)})\)).

    Using the finite-subadditivity of the large-deviation upper bound,
    \begin{multline*}
        \limsup_{N\to\infty} \frac{1}{N}\log \mathbb{P}^{N,K}(C)\\
        \le \max\Bigl\{
            \max_{1\le i \le k} \; \limsup_{N\to\infty} \frac{1}{N}\log \mathbb{P}^{N,K}(V_{\pi^{(i)}}), \;
            \max_{1\le j \le l} \; \limsup_{N\to\infty} \frac{1}{N}\log \mathbb{P}^{N,K}(U_{\tilde{\pi}^{(j)}})
        \Bigr\} \\
        \le \max\Bigl\{
            - \min_{1\le i \le k} \bigl\{\iota_{\nu}(\pi^{(i)}_0) + \mathcal{R}^K_{\mathrm{upper}}(\pi^{(i)})\bigr\} + \delta, \;
            -M
        \Bigr\} \\
        \le \max\Bigl\{
            - \inf_{\pi \in C} \bigl\{\iota_{\nu}(\pi_0) + \mathcal{R}^K_{\mathrm{upper}}(\pi)\bigr\} + \delta, \;
            -M
        \Bigr\}.
    \end{multline*}

    Letting \(M \to \infty\) and then \(\delta \to 0\) yields
    \[
        \limsup_{N\to\infty} \frac{1}{N}\log \mathbb{P}^{N,K}(C)
        \le - \inf_{\pi \in C} \bigl\{\iota_{\nu}(\pi_0) + \mathcal{R}^K_{\mathrm{upper}}(\pi)\bigr\}.
    \]

    Together with the exponential tightness established in Lemma~\ref{lemma:exp tight}, this upper bound for compact sets extends to all closed subsets of \(\mathscr{D}_T(\mathscr{M}_f^+(E))\), completing the proof of the upper bound.
\end{proof}

\subsection{Proof of the lower bound}

\begin{proof}[Proof of Theorem~\ref{thm:LDP lower}]
Fix $\pi\in O\cap H_0^K[\nu]$ and
$\eta\in\mathcal A^K_{\rm uniq}(\pi)$.  Write
$Q_N:=\mathbb P^{N,\hat\eta K}$ and choose an open neighborhood
$B$ of $\pi$ whose closure is contained in $O$.  By
Theorem~\ref{thm:LLN} applied to $\hat\eta K$, every subsequential limit
under $Q_N$ solves the tilted
Flory equation.  The uniqueness in the definition of
$\mathcal A^K_{\rm uniq}(\pi)$ therefore gives
\begin{equation}\label{eq:tilted-concentration}
Q_N(B)\longrightarrow1.
\end{equation}

Set $c_\eta:=\inf_{X_T}\eta>0$. Thus $\log\eta$ and
$\tau^*(\eta-1)=\eta\log\eta-\eta+1$ are bounded continuous functions.
After shrinking $B$ if necessary, we have
\begin{equation}\label{eq:lower-cost-continuity}
\left|\langle\tau^*(\eta-1),\Lambda_\rho^{K,T}\rangle
-\langle\tau^*(\eta-1),\Lambda_\pi^{K,T}\rangle\right|\leq\delta
\end{equation}
for every path $\rho\in B\cap\mathscr D_T^a$ reachable by the
$N$-particle system from $\mu_0^N$, for all sufficiently large $N$.
Indeed, failure of this assertion would produce $N_k\uparrow\infty$ and
reachable $\rho^k\to\pi$ violating \eqref{eq:lower-cost-continuity}.
The continuity of $\pi$, Assumption~\ref{assu:initial-2}(c), and
Proposition~\ref{prop:main continuous proposition}, applied to the
bounded continuous test function $\tau^*(\eta-1)$, give a contradiction.

Under $Q_N$, let
\[
Z_N:=\int_{[0,T]\times E^2\times E}
\log\eta(t,x,y,z)\,
\widetilde{\mathcal N}^{N,\hat\eta K}
(\diff t\diff x\diff y\diff z).
\]
Its predictable quadratic variation satisfies
\[
\mathbb E^{Q_N}[Z_N^2]
=N\mathbb E^{Q_N}
\langle\eta(\log\eta)^2,\kappa^{N,K}[\mu^N]\rangle
\leq C_\eta NaT.
\]
Consequently, $Q_N(|Z_N|>N\delta)\to0$. Conservation of $\phi$ in the
finite system and Lemma~\ref{lemma: cross term to 0}, applied under the
tilted law (using $K\leq c_\eta^{-1}\hat\eta K$), also give
\begin{equation}\label{eq:kappa-lambda-cost}
Q_N\left(\left|
\langle\tau^*(\eta-1),\Lambda_{\mu^N}^{K,T}\rangle
-\langle\tau^*(\eta-1),\kappa^{N,K}[\mu^N]\rangle
\right|>\delta\right)\longrightarrow0.
\end{equation}
Let $V_N$ be the intersection of the complementary events in the last two
estimates. By \eqref{eq:tilted-concentration},
$Q_N(B\cap V_N)\to1$.

Because $\eta$ is bounded away from zero, Lemma~\ref{lemma:coupling}
may be used in the reverse direction. On $B\cap V_N$,
\[
\begin{aligned}
\frac{\diff\mathbb P^{N,K}}{\diff Q_N}
&=\exp\left\{-Z_N-N
\langle\tau^*(\eta-1),\kappa^{N,K}[\mu^N]\rangle\right\}\\
&\geq\exp\left\{-N\langle\tau^*(\eta-1),
\Lambda_\pi^{K,T}\rangle-3N\delta\right\}.
\end{aligned}
\]
It follows that
\[
\liminf_{N\to\infty}\frac1N\log\mathbb P^{N,K}(O)
\geq-\langle\tau^*(\eta-1),\Lambda_\pi^{K,T}\rangle-3\delta.
\]
Letting $\delta\downarrow0$, then taking the infimum first over
$\eta\in\mathcal A^K_{\rm uniq}(\pi)$ and then over
$\pi\in O\cap H_0^K[\nu]$, proves the assertion.
\end{proof}

\subsection{Dual representation of the upper rate}

\begin{proof}[Proof of Theorem~\ref{thm:another representation}]
The moment assumption and $\bar K,\phi\leq\phi_0$ imply that
$\Lambda:=\Lambda_\pi^{K,T}$ is a finite positive measure. Since
$e^u-1=u+\tau(u)$, definition~\eqref{eq:upper-rate-compact} becomes
\begin{equation}\label{eq:upper-rate-tau}
\mathcal R^K_{\rm upper}(\pi)
=\sup_{g\in\mathscr C_{\bar E}^{1,0}}
\left\{\gamma_T^K(\pi,g)-
\langle\tau(L[\bar g]),\Lambda\rangle\right\}.
\end{equation}
Let $\mathcal G=\mathscr C_{\bar E}^{1,0}$, equipped with the norm
$\|g\|_{\mathcal G}=\|\bar g\|_\infty+
\|\partial_t\bar g\|_\infty$, and define the continuous linear map
\[
A:\mathcal G\longrightarrow\mathscr C(X_T),
\qquad Ag=L[\bar g].
\]
The moment assumption makes $b(g):=\gamma_T^K(\pi,g)$ a continuous
linear functional on $\mathcal G$.  Also set
\[
F(h):=\int_{X_T}\tau(h)\,\diff\Lambda,
\qquad h\in\mathscr C(X_T).
\]
Because $\Lambda$ is finite and $\tau$ is finite and continuous on
$\mathbb R$, $F$ is a finite continuous convex functional on the Banach
space $\mathscr C(X_T)$.  Thus \eqref{eq:upper-rate-tau} is
\[
\mathcal R^K_{\rm upper}(\pi)
=\sup_{g\in\mathcal G}\{b(g)-F(Ag)\}.
\]

The dual of $\mathscr C(X_T)$ is the space of finite signed Radon
measures.  If $\theta=z\Lambda+\theta^s$ is the Lebesgue decomposition,
the conjugate of $F$ is
\begin{equation}\label{eq:integral-conjugate}
F^*(\theta)=\int_{X_T}\tau^*(z)\,\diff\Lambda
+\int_{X_T}\tau^\infty\!\left(
\frac{\diff\theta^s}{\diff|\theta^s|}\right)\diff|\theta^s|,
\end{equation}
where $\tau^\infty$ is the recession function.  Replacing $X_T$ by the
compact support of $\Lambda$ if necessary, this is the
integral-functional conjugacy theorem of \cite[Theorem~5]{rock2}.  Here
\[
\tau^\infty(u)=
\begin{cases}
+\infty,&u>0,\\
0,&u=0,\\
-u,&u<0.
\end{cases}
\]
Since the polar density of a real signed measure equals $1$ or $-1$
$|\theta^s|$-almost everywhere, \eqref{eq:integral-conjugate} is exactly
$\mathcal J_\Lambda(\theta)$.

Finally, Fenchel--Rockafellar duality applied to $F\circ A-b$ gives
\[
\sup_{g\in\mathcal G}\{b(g)-F(Ag)\}
=\inf\{F^*(\theta):A^*\theta=b\}.
\]
There is no duality gap because $F$ is continuous at $A0=0$; this is the
continuity qualification in the Fenchel duality theorem
\cite[Section~31]{rock-convex}.  The constraint $A^*\theta=b$ is
precisely the defining constraint of
$\mathcal O_{\mathcal M}^K[\pi]$.  This proves
\eqref{eq:rate-representation}, including the case in which either side
is infinite.  Restricting the same formula to $\theta=z\Lambda$ gives
the stated absolutely continuous specialization and the finiteness
criterion.
\end{proof}



\section*{Acknowledgments}
This work is supported by the National Key R\&D Program of China (No. 2022YFA1006500) and by the National Natural Science Foundation of China (No. 12401171).
\bibliographystyle{amsplain}
\bibliography{ref}
\bigskip


\end{document}